\documentclass[11pt]{amsart}
\usepackage[margin=1in]{geometry}

\usepackage{amsmath, amsfonts, amsthm, amssymb}
\usepackage[shortlabels]{enumitem}
\usepackage{graphicx}
\usepackage{mathrsfs}
\usepackage{tikz}
\usepackage{xcolor}
\usepackage{hyperref}
\setlist{itemsep=2pt, topsep=4pt}
\definecolor{figblue}{RGB}{91,155,213}    % cornflower blue
\definecolor{editblue}{RGB}{31,78,121}    % darker revision blue
\definecolor{figorange}{RGB}{237,125,49}  % warm orange
\definecolor{figgold}{RGB}{255,192,0}     % amber gold
\definecolor{figgreen}{RGB}{112,173,71}   % sage green

\def\Vol{\mathrm{Vol}}

\def\tr{\mathrm{Tr}}
\def\Ric{\mathrm{Ric}}

\def\Len{\mathrm{Len}}

\def\diam{\mathrm{diam}}
\def\div{\mathrm{div}}

\renewcommand{\Re}{\mathrm{Re}}

\def\CC{\mathbb{C}}

\def\RR{\mathbb{R}}

\def\HH{\mathbf{H}}
\def\Heis{\mathbb{H}}

\def\II{\mathrm{I\!I}}

\def\eps{\varepsilon}

\newcommand{\mip}{\mathrm{m}}

\def\bJ{\mathbf{J}}
\def\mm{\mathsf{m}}

\def\cW{\mathcal{W}}

\def\rF{\mathrm{F}}
\def\rU{\mathrm{U}}
\def\cD{\mathcal{D}}

\def\sA{\mathscr{A}}

\def\CD{\mathrm{CD}}

\def\rv{\mathrm{v}}
\def\rx{\mathrm{x}}
\def\ry{\mathrm{y}}
\def\rz{\mathrm{z}}

\def\la{\langle}
\def\ra{\rangle}

\newtheorem{theorem}{Theorem}[section]
\newtheorem{Theorem}{Theorem}
\newtheorem{definition}[theorem]{Definition}

\newtheorem{lemma}[theorem]{Lemma}

\newtheorem{proposition}[theorem]{Proposition}
\newtheorem*{proposition*}{Proposition}
\newtheorem{corollary}[theorem]{Corollary}

\theoremstyle{remark}
\newtheorem{remark}[theorem]{Remark}
\newtheorem*{remark*}{Remark}

\makeatletter
\providecommand*{\toclevel@Theorem}{0}
\providecommand*{\toclevel@theorem}{0}
\providecommand*{\toclevel@lemma}{0}
\providecommand*{\toclevel@proposition}{0}
\providecommand*{\toclevel@corollary}{0}
\providecommand*{\toclevel@remark}{0}
\@ifundefined{toclevel@remark*}{\expandafter\def\csname toclevel@remark*\endcsname{0}}{}
\makeatother

\DeclareMathAlphabet{\mathcal}{U}{eus}{m}{n}

\title{Magnetic Brunn--Minkowski inequalities}
\author{Rotem Assouline}
\date{}

\begin{document}

\begin{abstract}
	We study Minkowski averages on Riemannian manifolds 
	in which the interpolation is by action-minimizing magnetic geodesics 
	with respect to a given magnetic potential. 
	We establish equivalence between Brunn--Minkowski inequalities
	for this operation and lower bounds on a magnetic Ricci curvature. 
	We then discuss various examples,
	including natural magnetic fields on K\"ahler and Sasakian manifolds,
	and prove a sharp, undistorted Brunn--Minkowski inequality
	for contact magnetic geodesics on the Heisenberg group.
	We also observe that closed magnetic potentials from different cohomology classes
	may give rise to different geodesic Minkowski averages.
\end{abstract}

\maketitle

\section{Introduction}

	The Brunn--Minkowski inequality asserts that for every pair $A_0,A_1 \subseteq \RR^n$ of nonempty Borel sets and every $0 \le \lambda \le 1$,
	the \emph{Minkowski $\lambda$-average} $A_\lambda = (1-\lambda)A_0 + \lambda A_1$ of $A_0$ and $A_1$ satisfies
	\begin{equation}\label{eq:brunn-minkowski}
		\Vol(A_\lambda)^{1/n} \ge (1-\lambda)\cdot\Vol(A_0)^{1/n} + \lambda\cdot\Vol(A_1)^{1/n},
	\end{equation}
	where $\Vol$ denotes Lebesgue measure.
	The Brunn--Minkowski inequality is one of the cornerstones of convex geometry~\cite{Schneider14}; for a survey of various applications and generalizations of the Brunn--Minkowski inequality, see~\cite{Gard}. 

	A generalization of the Brunn--Minkowski inequality to Riemannian manifolds was discovered by Cordero-Erausquin,
	McCann and Schmuckenschl\"ager~\cite{CMS} and Sturm~\cite{Sturm2},
	revealing an intimate connection to curvature.
	The Minkowski average $A_\lambda$ of a pair of subsets $A_0,A_1$ of a Riemannian manifold is defined to be the set of $\lambda$-midpoints of minimizing geodesics joining $A_0$ to $A_1$.
	With this definition, inequality~\eqref{eq:brunn-minkowski} holds verbatim for every pair of Borel sets of positive measure,
	provided that the metric is complete and has nonnegative Ricci curvature.
	
	Nonnegativity of the Ricci curvature is necessary for~\eqref{eq:brunn-minkowski} to hold for all pairs $A_0,A_1$.
	Without this assumption, a correction in the form of \emph{distortion coefficients} multiplying the terms on the right hand side is needed.
	These coefficients can be controlled under lower Ricci curvature bounds; in fact,
	any lower bound on the Ricci curvature can be characterized by a suitable \emph{distorted} Brunn--Minkowski inequality~\cite{Sturm2,Vil,MPR}.
	The Brunn--Minkowski inequality for geodesic interpolation has since been extended to other settings,
	including metric measure spaces~\cite{Sturm2,LV,CM}, Finsler manifolds~\cite{Oh09},
	sub-Riemannian manifolds~\cite{BKS18,BR,BMR}, and Lorentzian manifolds~\cite{BrMc,CM24}.
	In all these examples, the interpolating curves are geodesics of a metric. 

	The goal of this paper is to present a new class of Brunn--Minkowski inequalities,
	in which the interpolating curves are instead \emph{magnetic geodesics}, i.e.
	minimizers of a functional of the form $\Len - \int\eta$,
	where $\Len$ denotes length with respect to a Riemannian metric and $\eta$ is a given one-form.
	We show that magnetic Brunn--Minkowski inequalities on Riemannian manifolds characterize lower bounds on a certain \emph{magnetic Ricci curvature},
	in complete analogy to the metric case. 

	\medskip
	Let us highlight two examples of such magnetic Brunn--Minkowski inequalities.
	Both share the property that they take the form~\eqref{eq:brunn-minkowski} while their geodesic counterparts involve distortion coefficients.
	Other examples can be found in Section~\ref{sec:examples}.

	In~\cite{AK}, the author and Klartag proved a Brunn--Minkowski inequality in which the interpolating curves are unit-speed horocycles in the hyperbolic plane $\HH$; a generalization to complex hyperbolic space of arbitrary dimension was given in~\cite{Ass25}. A \emph{horocycle} in the hyperbolic plane is a curve of constant signed geodesic curvature $1$ with respect to a fixed orientation. More generally, in complex hyperbolic space $\CC\HH^d$, horocycles are solutions to the ordinary differential equation $\nabla_{\dot\gamma}\dot\gamma = \bJ\dot\gamma$, where $\bJ$ is the complex structure; for more details see Section~\ref{sec:horo}. The horocyclic Minkowski average $A_\lambda$ of two sets $A_0,A_1\subseteq\CC\HH^d$ is defined to be the set of $\lambda$-midpoints of horocycles joining $A_0$ to $A_1$.

	\begin{Theorem}[Horocyclic Brunn--Minkowski inequality~\cite{AK,Ass25}]\label{thm:horocyclic-bm}
		For every pair $A_0,A_1 \subseteq \CC\HH^d$ of Borel sets of positive measure and every $0 \le \lambda \le 1$,
		the horocyclic Minkowski average $A_\lambda$ of $A_0$ and $A_1$ satisfies
		\begin{equation}\label{eq:horocyclic-bm}
			\Vol(A_\lambda)^{1/n} \ge (1-\lambda)\cdot\Vol(A_0)^{1/n} + \lambda\cdot\Vol(A_1)^{1/n},
		\end{equation}
		where $n=2d$ is the real dimension of $\CC\HH^d$.
		If $A_0$ and $A_1$ are concentric balls then equality holds.
	\end{Theorem}

	The next example, which appears here for the first time,
	is a Brunn--Minkowski inequality for contact magnetic geodesics on the Heisenberg group $\Heis^m$.
	For $A_0,A_1 \subseteq \Heis^m$ and $0 \le \lambda \le 1$,
	the \emph{contact magnetic Minkowski $\lambda$-average} $A_\lambda$ of $A_0$ and $A_1$ is the set of $\lambda$-midpoints of minimizing magnetic geodesics joining $A_0$ to $A_1$, where the metric is the left-invariant Riemannian metric on $\Heis^m$ and the magnetic potential $\eta$ is the standard contact form; see Section~\ref{sec:heisenberg-example} and Section~\ref{sec:Heisenberg}.

	\begin{Theorem}[Contact magnetic Brunn--Minkowski inequality on the Heisenberg group]\label{thm:critical-heisenberg-bm}
	For every pair $A_0,A_1\subseteq\Heis^m$ of Borel sets of positive measure and every $0 \le \lambda \le 1$,
	the contact magnetic Minkowski $\lambda$-average $A_\lambda$ of $A_0$ and $A_1$ satisfies
	\begin{equation}\label{eq:critical-heisenberg-bm}
		\Vol(A_\lambda)^{1/n}
		\ge
		(1-\lambda)\cdot\Vol(A_0)^{1/n}
		+
		\lambda\cdot\Vol(A_1)^{1/n},
	\end{equation}
	where $n=2m+1$.
	\end{Theorem}
	The contact magnetic Minkowski average of two sets in $\Heis^m$ can have infinite volume,
	even if the sets are arbitrarily small; see Lemma~\ref{lem:P-infinite-midpoints}. Still,
	Theorem~\ref{thm:critical-heisenberg-bm} is sharp,
	see Lemma~\ref{lem:heisenberg-cylinder-horizontal-midpoints}.
	Existing Brunn--Minkowski inequalities on $\Heis^m$ include the multiplicative Brunn--Minkowski inequality~\cite{LM05,bobkov2011brunn,Poz19} and the sub-Riemannian geodesic Brunn--Minkowski inequality~\cite{BKS18}. Since the Ricci curvature of the left-invariant Riemannian metric on $\Heis^m$ attains negative values, there is no undistorted Brunn--Minkowski inequality for geodesic interpolation on $\Heis^m$ with respect to that metric.

	Theorem~\ref{thm:horocyclic-bm} follows directly from Theorem~\ref{thm:main} below,
	which asserts equivalence between the magnetic Brunn--Minkowski inequality and nonnegativity of the magnetic Ricci curvature,
	under appropriate global assumptions.
	Theorem~\ref{thm:critical-heisenberg-bm} is more delicate since one of the global assumptions is not satisfied (namely,
	not every pair of points can be joined by a minimizing magnetic geodesic), and a limiting procedure is required.

	\medskip
	The rest of the paper is organized as follows.
	In Section~\ref{sec:magnetic-geodesics} we define magnetic geodesics and magnetic Ricci curvature,
	and state our main result. In Section~\ref{sec:examples} we give examples.
	In Sections~\ref{sec:ric-to-bm} and~\ref{sec:bm-to-ric} we prove the magnetic Brunn--Minkowski inequality and its converse,
	respectively. In Section~\ref{sec:Heisenberg} we prove Theorem~\ref{thm:critical-heisenberg-bm}.

	\medskip
	\textbf{Acknowledgements}.
	I would like to thank Bo'az Klartag for providing me with invaluable mentorship in my doctoral years,
	during which this work was conceived.
	I am supported by the Rothschild fellowship (Yad Hanadiv) and by the Fondation Sciences Math\'ematiques de Paris (FSMP).

	\medskip
	\textbf{Use of AI tools}.
	AI-assisted tools were used during the preparation of this manuscript for
	proofreading and editorial suggestions, for creating figures, and for checking
	computations and verifying proofs. No mathematical results or proofs were
	generated by AI.

\section{Magnetic geodesics, magnetic Ricci curvature, and magnetic Brunn--Minkowski}\label{sec:magnetic-geodesics}

	\begin{figure}[htbp]
		\centering
		\includegraphics[width=.8\linewidth]{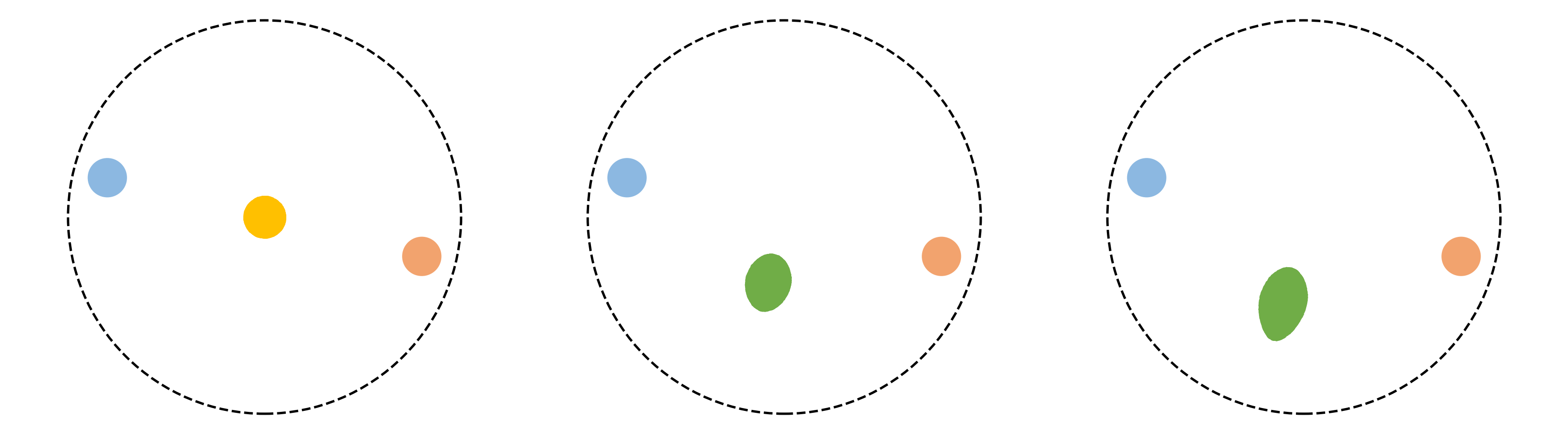}
		\caption{The magnetic Minkowski $\tfrac12$-average of two discs of radius $\tfrac{1}{10}$ in the unit disc, with the Euclidean metric and with the magnetic field $\Omega=d\eta=c\,dx\wedge dy$ for $c = 0$ (left), $c=0.85$ (middle) and $c=1$ (right).}
		\label{fig:euclidean-magnetic}
	\end{figure}
	
	\subsection{Magnetic geodesics and magnetic Minkowski averages}

	Let $(M,g)$ be a Riemannian manifold of dimension $n \ge 2$ and let $\Omega$ be a closed 2-form on $M$.
	A \emph{magnetic geodesic} is a solution $\gamma$ to the ordinary differential equation
	\begin{equation}\label{eq:lorentz-force}
		g\left(\nabla_{\dot\gamma}\dot\gamma,\,\cdot\,\right) = \Omega(\dot\gamma,\,\cdot\,),
	\end{equation}
	where $\nabla$ is the Levi-Civita connection of $g$.
	We will assume that the 2-form $\Omega$ is exact, i.e. there exists a 1-form $\eta$ such that
	\[
		\Omega = d\eta.
	\]

	We fix such $\eta$, and define a \emph{minimizing magnetic geodesic} to be a unit-speed curve $\gamma$ minimizing
	\[
		\Len[\gamma] - \int_\gamma\eta
	\]
	among all piecewise-$C^1$ curves joining its endpoints.
	We consider constant curves to be minimizing magnetic geodesics as well.
	Minimizing magnetic geodesics are indeed solutions to~\eqref{eq:lorentz-force},
	which is the Euler-Lagrange equation for the Lagrangian $L = \tfrac12g-\eta$ (viewing both $g$ and $\eta$ as functions on $TM$).
	It is important to note that the class of minimizing magnetic geodesics depends on the choice of $\eta$,
	which is not unique. Adding an exact one-form to $\eta$ does not change this class,
	but adding a one-form which is closed but not exact does,
	though the Euler-Lagrange equation~\eqref{eq:lorentz-force}, which characterizes stationary curves,
	remains the same. We will come back to this point in Section~\ref{sec:closed},
	where we consider the case $\Omega = 0$.

	In the sequel, we make the following global assumptions:
	\begin{enumerate}[(I)]
		\item For every pair of points $x,y \in M$ there exists at least one minimizing magnetic geodesic $\gamma:[0,\ell]\to M$ joining $x$ to $y$,
		i.e. $\gamma(0) = x$ and $\gamma(\ell) = y$.
		\item For every \emph{closed}, nonconstant, piecewise-$C^1$ curve $\gamma$,
			\[
				\int_\gamma \eta < \Len[\gamma].
			\]
		\item For every compact set $A\subseteq M$ there exists a compact set $\tilde A \supseteq A$ such that every minimizing magnetic geodesic with endpoints in $A$ lies entirely in $\tilde A$.
	\end{enumerate}

	\begin{remark}\label{rem:global-assumptions}
		Assumption (II), which is closely related to the Ma\~n\'e critical value of the magnetic Lagrangian,
		is clearly satisfied when there is a smooth function $f$ such that $|\eta + df|_g < 1$; in fact,
		when $M$ is compact, the two conditions are equivalent~\cite{BP}. Moreover, in the compact case,
		assumption (I) follows from assumption (II)~\cite[Theorem X]{CDI}, and assumption (III) is vacuous.
		In the noncompact case, if $g$ is complete and there exist a smooth function $f$ and $\eps > 0$ such that $|\eta + df|_g \le 1 - \eps$,
		then conditions (I) and (III) are satisfied. Indeed,
		in this case minimizing magnetic geodesics coincide with minimizing geodesics of a Finsler metric~\cite{Cont,IS},
		which is complete by the uniform boundedness assumption. However,
		this is not a necessary condition: in the case of horocycles in the hyperbolic plane,
		conditions (I)-(III) hold but no such $f$ exists.
	\end{remark}

	\begin{remark}\label{rem:local-convexity}
		Assumptions (I)-(III) imply that every point $x\in M$ has a neighborhood $U$ such that every pair of points in $U$ is joined by a unique minimizing magnetic geodesic, which is contained in $U$; see~\cite[Lemma~2.9]{Ass25}.
	\end{remark}

	\subsection{Magnetic Ricci curvature}

	Let $SM : = \left\{v \in TM \, \, \colon \, \,
	g(v,v) = 1 \right\}$ denote the unit tangent bundle.
	The \emph{magnetic Ricci curvature} is the function $\Ric_\Omega : SM \to \RR$ given by
	\[
		\Ric_\Omega(v) : = \Ric(v) - (\delta\Omega)(v) + \tfrac12|\iota_v\Omega|^2 + \tfrac14|\Omega|^2, \qquad v \in SM.
	\]

	Here $\Ric$ is the Ricci curvature of the metric $g$ with $\Ric(v):=\Ric(v,v)$,
	the codifferential $\delta$ is defined on $k$-forms by $\delta = (-1)^k \star^{-1} d \,
	\star$ where $\star$ is the Hodge star, the interior product $\iota$ is defined by $\iota_v\Omega = \Omega(v,\cdot)$ and $|\Omega|^2$ is the (unnormalized) squared Frobenius norm $|\Omega|^2 = \sum_{i,j}\Omega(e_i,e_j)^2$ where $(e_i)$ is any orthonormal frame.
	
	The terms in the definition of $\Ric_\Omega$ have different degrees of homogeneity---the first and third are quadratic,
	the second is linear and the fourth is constant on each fiber.
	But since we view all of them as real-valued functions on $SM$, adding them makes sense.

	The definition of $\Ric_\Omega$ is motivated by the magnetic Bochner formula (Proposition~\ref{prop:magnetic-bochner}).
	The first two terms are the trace of a ``magnetic Riemann curvature'' appearing in the magnetic Jacobi equation; see Lemma~\ref{lem:magnetic-jacobi}.
	The other two terms are there as a consequence of the failure of the Gauss lemma in the magnetic setting: if $V$ is a unit vector field whose integral curves are magnetic geodesics, then by~\eqref{eq:lorentz-force}, the covariant derivative $\nabla V$ has a component in the direction of $V$.

	The magnetic Ricci curvature $\Ric_\Omega$ is not, in itself, a new object.
	It appears in the literature under different guises,
	mainly in works on dynamics of magnetic flows~\cite{Gouda,Gro,Woj,Asse,AT,BA}. Furthermore,
	it is in fact a Riemannian Ricci curvature---once one adds an extra dimension and considers the Kaluza-Klein metric on the product,
	evaluated at an appropriate lift; see~\cite[Appendix A]{AM}.
	For works on horizontal Ricci curvature on sub-Riemannian and foliated manifolds,
	see~\cite{ABR17,BW14,LLZ15,GT16}. After reparametrization,
	magnetic geodesics become geodesics of the Finsler metric $F(v) = \sqrt{g(v,v)} - \eta(v)$. However,
	because of this reparametrization,
	the magnetic Ricci curvature does not coincide with the Finslerian Ricci curvature of $F$.

	\subsection{Statement of main result}
	
	Our main result is a characterization of nonnegativity of $\Ric_\Omega$ in terms of a magnetic Brunn--Minkowski inequality.
	
	Given $A_0,A_1 \subseteq M$ and $0 \le \lambda \le 1$,
	we define the \emph{magnetic Minkowski $\lambda$-average} of $A_0$ and $A_1$ to be
	\[
		A_\lambda :=
		\left\{
		\quad
		\gamma(\lambda \ell)
		\quad
		\colon
		\quad
		\begin{array}{c}
			\gamma:[0,\ell] \to M \text{ is a minimizing magnetic}\\
			\text{geodesic, } \,\,\gamma(0) \in A_0, \quad \gamma(\ell) \in A_1
		\end{array}
		\right\}.
	\]

	\begin{Theorem}[Magnetic Brunn--Minkowski inequality]\label{thm:main}
		The following are equivalent:
		\begin{enumerate}[(i)]
			\item $\Ric_\Omega \ge 0$ on $SM$.
			\item For every pair $A_0,A_1 \subseteq M$ of Borel sets of positive measure and every $0 \le \lambda\le 1$,
			\[
				\Vol(A_\lambda)^{1/n} \ge (1-\lambda)\cdot\Vol(A_0)^{1/n} + \lambda\cdot\Vol(A_1)^{1/n},
			\]
			where $\Vol$ denotes Riemannian volume and $n = \dim M$.
		\end{enumerate}
	\end{Theorem}

	As in the Riemannian case~\cite{Sturm2,Vil,MPR},
	arbitrary lower bounds on $\Ric_\Omega$ can be characterized by a \emph{distorted} Brunn--Minkowski inequality.
	The full characterization will be stated in Sections~\ref{sec:ric-to-bm} and~\ref{sec:bm-to-ric}.
	For simplicity, we do not treat here the more general setting in which
	$\Vol$ is replaced by an arbitrary smooth measure on
	$M$ and $1/n$ by a different exponent,
	although our methods extend to that case. 

	The implication (i)$\implies$(ii) in Theorem~\ref{thm:main} is a special case of a more general result which was proved in~\cite{Ass25} (and which encompasses arbitrary reference measures and more general exponents). Despite the aforementioned connections to Riemannian and Finslerian geometry, this implication cannot (at least not in any obvious sense) be achieved merely by reducing to Riemannian optimal transport, nor to any other metric optimal transport. In~\cite{AK} and~\cite{Ass24}, projective Finsler metrizability of magnetic flows was used; still, the difference in parametrization means that Finslerian optimal transport by itself is not sufficient in order to establish the correct inequalities. Our proof employs Klartag's needle decomposition technique using $L^1$ optimal transport~\cite{Kl} with cost induced by the magnetic Lagrangian, and relies on a needle decomposition result established in~\cite{Ass25}, see Proposition~\ref{prop:magnetic-needle}.

	The implication (ii)$\implies$(i) in Theorem~\ref{thm:main} is new.
	A proof in the Riemannian case was given in~\cite{MPR}. Our proof, though conceptually similar,
	is shorter and makes no mention of optimal transport.

\section{Examples}\label{sec:examples}

	In this section we compute $\Ric_\Omega$ and check conditions (I)-(III) in some particular cases.

	\subsection{Surfaces}
	
		When $M$ is an oriented Riemannian surface, every two-form $\Omega$ can be written as
		\[
			\Omega = \kappa \, \Omega_g
		\]
		where $\Omega_g$ is the Riemannian area form and $\kappa : M \to \RR$ is a smooth function;
		if $\gamma$ is a magnetic geodesic, then $\kappa \circ \gamma$ is its signed geodesic curvature.
		The formula for $\Ric_\Omega$ in this case simplifies to
		\[
			\Ric_\Omega = K + \kappa^2 + \star \, d\kappa,
		\]
		where $K: M \to \RR$ is the Gauss curvature and $\star$ is the Hodge star (since $M$ is two-dimensional,
		$\star\, d\kappa$ is also a one-form). It follows that 
		\[
			\Ric_\Omega \ge 0 \qquad \iff \qquad K + \kappa^2 - |d\kappa| \ge 0.
		\]
		If $M$ is not oriented then $\kappa$ is defined only up to a sign,
		but $\Ric_\Omega$ is still well defined.
		
		If $M$ is the hyperbolic plane and $\kappa\equiv 1$,
		then magnetic geodesics are oriented horocycles.
		We discuss the case of horocycles in more detail in the next example;
		see also~\cite[Section 7]{Ass25}. For a treatment of the two-dimensional case,
		including additional examples, see~\cite[Section 5]{Ass24}.

	\subsection{K\"ahler manifolds}

		Let $(M,g,\omega)$ be a K\"ahler manifold of complex dimension $d = n/2$, let $c \in \RR$ and let
		\[
			\Omega_c := c \cdot \omega.
		\]
		Magnetic geodesics for the two-form $\Omega_c$ are solutions to
		\[
			\nabla_{\dot\gamma}\dot\gamma = c\cdot \bJ\dot\gamma,
		\]
		where $\bJ$ is the complex structure. Moreover,
		\[
			\delta\Omega_c(v) = 0, \qquad |\Omega_c|^2 = c^2 \cdot 2d \qquad \text{ and } \qquad |\iota_v\Omega_c|^2 = c^2, \qquad v \in SM,
		\]
		whence the magnetic Ricci curvature takes the simple form
		\begin{equation}\label{eq:KahlerRic}
			\Ric_{\Omega_c} = \Ric + c^2 \cdot \frac{d+1}{2}.
		\end{equation}
		Curvature-type properties of such magnetic fields were studied extensively by Adachi;
		see~\cite{Ada97,Ada06,Ada12,BA} and references therein. 
		In particular, it was shown in~\cite{Ada12} that condition (I) holds for K\"ahler magnetic fields when the metric is complete and its sectional curvature is bounded above by $-c^2$. 

		\subsubsection{Circular arcs in complex Euclidean space}

			Let $M = U \subseteq \CC^d$ be a convex domain in complex Euclidean space with a smooth boundary,
			and assume $c > 0$. In this case $\Ric_{\Omega_c} \equiv c^2\cdot \tfrac{d+1}{2}$,
			and magnetic geodesics are precisely oriented circular arcs of radius $1/c$ contained in affine complex lines. 
			
			Assume that for every affine complex line $L \subseteq \CC^d$ transverse to $\partial U$,
			the curvature of the curve $U \cap L$ is bounded below by $c$.
			In terms of the second fundamental form, this is equivalent to
			\[
				\II(v) \ge c \cdot |v|\cdot|\left<v,i\nu\right>| \qquad \text{for every $v \in T\partial U$,}
			\]
			where $\II$ is the second fundamental form of $\partial U$.
			Then conditions \textup{(I)} and \textup{(III)} hold; indeed,
			for every $x,y \in U$ we may take $L$ to be the affine complex line through $x$ and $y$,
			and then by our assumption the section $U \cap L$ must contain in its interior both circular arcs of radius $1/c$ joining $x$ to $y$.
			Furthermore, if $0<t < 1$ is chosen so that $x,y \in t\cdot U$ then the entire magnetic geodesic joining $x$ to $y$ is contained in $t \cdot U$,
			which proves \textup{(III)}.

			Condition \textup{(II)} may be ensured by additional hypotheses. For example,
			if $U$ is contained in a ball of radius $<2/c$ centered at the origin,
			then the primitive $\eta_c = (ic/4)\cdot \sum_j \left( z^j\, d\bar z^j - \bar z^j\,
			dz^j \right)$ satisfies $|\eta_c|^2 = c^2|z|^2/4 < 1$ on $U$. Alternatively, writing $z=x+iy$,
			if $U$ is contained in a slab of the form $\{|y|<1/c\}$,
			then the primitive $\tilde\eta_c:=-c \cdot \sum_j y_j\,
			dx_j$ satisfies $|\tilde\eta_c|^2 < 1$ on $U$.
			Note that these are only sufficient conditions for \textup{(II)}. 
			
			In particular, if $U$ is a complex ellipsoid all of whose (complex) semi-axes are $\le 1/c$,
			or if all the principal curvatures of $\partial U$ are $\ge c$, then conditions (I)-(III) hold.
			
		\subsubsection{Horocycles in complex hyperbolic space}\label{sec:horo}

			Let $M = \CC \HH^d$ denote complex hyperbolic space,
			normalized to have holomorphic sectional curvature $\equiv -1$, and set $c = 1$.
			Magnetic geodesics in this case are horocycles contained in totally geodesic copies of the hyperbolic plane.
			The Ricci curvature of $\CC\HH^d$ is $\Ric = -\tfrac{d+1}{2}\cdot g$,
			so formula~\eqref{eq:KahlerRic} implies that $\Ric_{\Omega_1} \equiv 0$.
			Assumptions (I)-(III) in this case were verified in~\cite[Section 7.1]{Ass25}.
			Thus Theorem~\ref{thm:horocyclic-bm} follows from Theorem~\ref{thm:main}. 

			The fact that the magnetic Ricci curvature vanishes identically suggests that there should be equality cases in the horocyclic Brunn--Minkowski inequality. By computing the horocyclic Minkowski average of two concentric balls, we can see that this is indeed the case:
			\begin{lemma}
				If $A_0,A_1 \subseteq \CC\HH^d$ are concentric balls then equality holds in~\eqref{eq:horocyclic-bm}.
			\end{lemma}
			\begin{proof}
				By symmetry, the set $A_\lambda$ is a ball.
				The volume of a ball of radius $r > 0$ in $\CC\HH^d$ is given by $c_d\cdot \sinh^{2d}(r/2)$ where $c_d$ is a dimensional constant,
				see~\cite[Lemma~6.18]{Gra90}.
				Thus it remains to prove that the radii $r_0,r_\lambda,r_1$ of the balls $A_0,A_\lambda,A_1$ satisfy the relation
				\[
					\sinh(r_\lambda/2) = (1-\lambda)\sinh(r_0/2) + \lambda\sinh(r_1/2).
				\]
				Denote the common center of the balls by $x$.
				By the definition of the horocyclic Minkowski average,
				we need to show that if $\gamma$ is a unit-speed horocycle satisfying $d(\gamma(0),x) \le r_0$ and $d(\gamma(\ell),x) \le r_1$,
				then $\sinh(d(\gamma(\lambda \ell),x)/2) \le (1-\lambda)\sinh(r_0/2) + \lambda\sinh(r_1/2)$,
				and that equality is attained for some horocycle $\gamma$.
				This would in turn follow if we prove that for every horocycle $\gamma$,
				the function $t \mapsto \sinh(d(\gamma(t),x)/2)$ is convex, and that it is linear when $x$ lies on $\gamma$. 

				Let $p$ denote the orthogonal projection of $x$ to the complex geodesic $H$ containing $\gamma$.
				Then the geodesic joining $x$ to $p$ is perpendicular to $H$,
				whence $\triangle xp\gamma(t)$ is a totally real right triangle, i.e.
				it is contained in a totally geodesic surface of curvature $\equiv -1/4$.
				The hyperbolic law of cosines in such a surface implies that
				\begin{equation}\label{eq:horo-cosh}
				\cosh\!\left(\frac{d(x,\gamma(t))}{2}\right)
				=
				{a}\cdot
				\cosh\!\left(\frac{d(p,\gamma(t))}{2}\right)
				\qquad \text{where} \qquad
				{a} : = \cosh\!\left(\frac{d(x,p)}{2}\right).
				\end{equation}

				We may translate the domain of $\gamma$ so that $\gamma(0)$ is the point on $\gamma$ closest to $p$,
				and then we can write
				\begin{equation}\label{eq:horo-arcsinh}
					d(p,\gamma(t)) = 2\cdot\mathrm{arcsinh}\left(\frac{\sqrt{(y_0t)^2 + (y_0-1)^2}}{2\sqrt{y_0}}\right),
				\end{equation}
				where $y_0$ is the constant imaginary part of $\gamma$ when it is represented as a horizontal line in an upper half-plane coordinate system on $H$ in which $p = i$. Combining~\eqref{eq:horo-cosh},~\eqref{eq:horo-arcsinh} and the identities $\cosh(\mathrm{arcsinh}(u)) = \sqrt{1+u^2}$ and $\sinh(\mathrm{arccosh}(v)) = \sqrt{v^2-1}$, we obtain
				\begin{align*}
					\sinh\left(\frac{d(x,\gamma(t))}{2}\right)
					& =
					\sqrt{{a}^2\left(1 + \frac{(y_0t)^2 + (y_0-1)^2}{4y_0}\right)-1},
				\end{align*}
				which is evidently convex in $t$. Moreover, if $x$ lies on $\gamma$ then $p = x$,
				${a} = 1$ and $y_0 = 1$, and the formula simplifies to $|t|/2$. This finishes the proof.
			\end{proof}
	\subsection{Sasakian manifolds}

		An \emph{almost contact structure} on an $n = (2m+1)$-dimensional smooth manifold $M$ is a triple $(\eta,\xi,\Phi)$ consisting of a $1$-form $\eta$,
		a vector field $\xi$, and a $(1,1)$-tensor $\Phi$ such that
		\[
			\eta(\xi)=1 \qquad \text{ and } \qquad \qquad \Phi^2 = -\mathrm{Id} + \eta \otimes \xi.
		\]
		A Riemannian metric $g$ on $M$ is \emph{compatible} with $(\eta,\xi,\Phi)$ if
		\begin{equation*}
			g(\Phi X,\Phi Y) = g(X,Y) - \eta(X)\eta(Y)
		\end{equation*}
		for all vector fields $X,Y$.
		The quintuple $(M,g,\eta,\xi,\Phi)$ is then called an \emph{almost contact metric manifold}.
		An almost contact metric manifold is \emph{Sasakian} if its metric cone is K\"ahler; see~\cite{Bl} for more details.
		The main examples we have in mind are $S^n$ and the Heisenberg group,
		both of which will be discussed below.
		For works on magnetic geodesics on Sasakian manifolds see~\cite{CFG,DIMN,Ika22,MN}.
		Let $(M,g,\eta,\xi,\Phi)$ be a Sasakian manifold of dimension $n = 2m + 1$, let $c \in \RR$ and set
		\[
			\Omega_c := c \cdot d\eta.
		\]

		Magnetic geodesics for the magnetic field $\Omega_c$ will be called \emph{contact $c$-magnetic geodesics}.
		When $c=1$ we will call them simply \emph{contact magnetic geodesics}.
		The \emph{contact $c$-magnetic Minkowski $\lambda$-average} $A_\lambda$ of a pair of subsets $A_0,A_1 \subseteq M$ is the set of $\lambda$-midpoints of minimizing contact magnetic geodesics joining $A_0$ to $A_1$. 
		
		We have
		\[
			d\eta(X,Y) = 2\cdot g(X,\Phi Y), \qquad
			g(\Phi X,\Phi Y) = g(X,Y) - \eta(X)\eta(Y),
		\]
		for every pair $X,Y$ of vector fields, and
		\[
			\delta d\eta = 4m \cdot \eta;
		\]
		see~\cite[Chapter~6]{Bl}. Therefore
		\[
			|\Omega_c|^2 = 4c^2 \cdot|\Phi|^2 = 8mc^2, \qquad
			|\iota_v\Omega_c|^2 = 4c^2 \cdot|\Phi v|^2 = 4c^2 \cdot\left(1 - \eta(v)^2\right), \quad v \in SM
		\]
		and
		\[
			\delta\Omega_c = 4cm \cdot\eta.
		\]

		Combining the above formulae with the definition of the magnetic Ricci curvature we get
		\begin{equation}
			\Ric_{\Omega_c}
			= \Ric + 2c^2\cdot(m+1) - 4cm \cdot\eta  - 2c^2 \cdot \eta^2,
		\end{equation}
		where again both $\Ric$ and $\eta$ are viewed as functions on $SM$.

		Since $d\eta = 2g(\,\cdot\,,\Phi \,\cdot\,)$,
		the magnetic geodesic equation $g(\nabla_{\dot\gamma}\dot\gamma,\cdot) = \Omega_c(\dot\gamma,\cdot)$ reads
		\[
			\nabla_{\dot\gamma}\dot\gamma = -2c\,\Phi\dot\gamma.
		\]

		Using $\nabla\xi = -\Phi$ and $\Phi\xi = 0$, we obtain
		\[
			\frac{d}{dt}\eta(\dot\gamma)
			= g(\nabla_{\dot\gamma}\dot\gamma,\xi) + g(\dot\gamma,\nabla_{\dot\gamma}\xi)
			= -2c\,g(\Phi\dot\gamma,\xi) - g(\dot\gamma,\Phi\dot\gamma)
			= 0.
		\]

		Hence $\eta(\dot\gamma)$ is constant along every magnetic geodesic. In particular,
		if $\dot\gamma(0) \in \cD : = \ker\eta$ then $\gamma$ is a geodesic of the sub-Riemannian structure $(M,\cD,g\vert_{\cD})$; see~\cite[Lemma 6.7]{ABR17} or~\cite[Proposition 15]{Rum92}.

		\subsubsection{Odd-dimensional spheres}

			Take $M = S^n \subseteq \CC^{d+1}$ to be the unit sphere in complex $(d+1)$--dimensional Euclidean space,
			$n = 2d+1$, and
			\[
				\eta\vert_z := \Re\la iz,\cdot\ra _{\mathbb{C}^{d+1}} = \frac{i}{2}\sum_{j=1}^{d+1}\bigl(z^jd\bar z^j - \bar z^jdz^j\bigr), \qquad z \in S^{2d+1},
			\]
			so that
			\[
				\Omega_c = c \cdot i\sum_{j=1}^{d+1}dz^j\wedge d\bar z^j.
			\]
			Fibers of the Hopf fibration are integral curves of the Reeb vector field $\xi$,
			and since $\iota_\xi\Omega_c = 0$, they are both ordinary and magnetic geodesics.
			We have that $\Ric \equiv 2d$ whence
			\[
				\Ric_{\Omega_c} = 2d + 2c^2(d+1) - 4cd\eta - 2c^2\eta^2,
			\]
			so since $|\eta| = 1$,
			\[
				\Ric_{\Omega_c} \ge 2d+2c^2(d+1) -4cd - 2c^2 = 2d(1 - c)^2 \ge 0.
			\]
			On each fiber of $SM$, the minimum and maximum of $\Ric_{\Omega_c}$ are attained at the Reeb directions $\pm \xi$.

			A detailed description of contact magnetic geodesics on odd-dimensional spheres was given in~\cite{ABM}.
			It was shown that if $|c| < 1$,
			then every pair of points can be joined by a minimizing magnetic geodesic, i.e. assumption (I) holds.
			Since $|c\eta| < 1$ and $S^n$ is compact, assumptions (II) and (III) also hold. Thus we have the following:

			\begin{theorem}[Contact magnetic Brunn--Minkowski inequality in $S^{2d+1}$~\cite{Ass25}]
				Let $A_0,A_1 \subseteq S^{2d+1}$ be Borel sets of positive measure and let $0 \le \lambda \le 1$.
				For every $c \in (-1,1)$,
				the contact $c$-magnetic Minkowski $\lambda$-average $A_\lambda$ of $A_0$ and $A_1$ satisfies
				\[
					\Vol(A_\lambda)^{1/n} \ge (1-\lambda)\cdot \Vol(A_0)^{1/n} + \lambda\cdot\Vol(A_1)^{1/n},
				\]
				where $\Vol$ denotes the spherical volume measure and $n = 2d+1$.
			\end{theorem}
		\subsubsection{The Heisenberg group}\label{sec:heisenberg-example}

			Let $M = \Heis^m \cong \RR^{2m+1}$ be the Heisenberg group,
			endowed with the left-invariant metric
			\[
				g = \sum_{j=1}^m (dx^j)^2 + \sum_{j=1}^m (dy^j)^2 + \left(dz - \sum_{j=1}^m y^jdx^j\right)^2
			\]
			and the contact form
			\[
				\eta := dz - \sum_{j=1}^m y^jdx^j,
			\]
			in coordinates $(x^1,\dots,x^m,y^1,\dots,y^m,z) \in \RR^m\times\RR^m\times \RR$.
			The Ricci curvature of $g$ is given by
			\[
				\Ric = -\frac12g + \frac{m+1}{2}\eta^2,
			\]
			see for instance~\cite[Chapter~7]{Bl} or~\cite{BGM06} (note different normalizations).
			Moreover,
			\[
				d\eta(X,Y)=g(X,\Phi Y),\qquad
				\delta d\eta=m\eta,
			\]
			and consequently
			\[
				|\Omega_c|^2=2mc^2,\qquad
				|\iota_v\Omega_c|^2=c^2(1-\eta(v)^2),\qquad
				\delta\Omega_c=cm\eta.
			\]
			Note that in this example we are using a different normalization than in the general Sasakian case, namely $d\eta(X,Y)=g(X,\Phi Y)$ rather than $d\eta(X,Y)=2g(X,\Phi Y)$. Substituting the above identities into the definition of magnetic Ricci curvature, we get
			\begin{align*}
				\Ric_{\Omega_c}
				& = \frac m2(c-\eta)^2 + \frac12(c^2-1)(1-\eta^2).
			\end{align*}
			It follows that
			\[
				\Ric_{\Omega_c} \ge 0
				\qquad \iff \qquad
				|c| \ge 1.
			\]

			As in the previous example, $|c|=1$ is the critical case for connectivity by minimizing magnetic geodesics:

			\begin{lemma}\label{lem:heisenberg-conditions}
				Assumptions~\textup{(I)}--\textup{(III)} hold for $|c| < 1$,
				assumption~\textup{(II)} fails for $|c|>1$ and holds for $c = \pm 1$,
				and assumption~\textup{(I)} fails for $c = \pm 1$.
			\end{lemma}
			\begin{proof}
				If $|c|<1$ then the magnetic potential $c\eta$ has norm bounded uniformly away from 1,
				so assumptions~\textup{(I)}--\textup{(III)} hold by Remark~\ref{rem:global-assumptions}.

				Suppose that $|c|>1$ and let $\gamma$ denote a square of sidelength $a$ in the $(x^1,y^1)$-plane,
				with orientation chosen according to the sign of $c$. Then for $a$ sufficiently large,
				\[
					\Len[\gamma] = a + a + \sqrt{a^2 + a^4} + a < |c|\cdot a^2 = \int_\gamma c\eta,
				\]
				so assumption (II) fails in this case.
				
				Now assume that $|c|=1$. By symmetry it suffices to consider the case $c=1$.
				Since $\eta$ has unit norm,
				for every unit-speed curve $\gamma$ we have $\int_\gamma \eta \le \Len[\gamma]$,
				with equality if and only if $\dot\gamma \equiv \partial_z$.
				But the integral curves of $\partial_z$ are vertical lines, which are not closed.
				Therefore assumption~\textup{(II)} holds.
				
				Assumption~\textup{(I)} fails since there is no minimizing
				magnetic geodesic joining $(0,0,0)$ to $(0,0,z_1)$ when
				$z_1<-\pi$, see
				Lemma~\ref{lem:heisenberg-minimizers}.  
			\end{proof}
	
	\subsection{Closed magnetic potentials}\label{sec:closed}

		The Riemannian Minkowski average was defined in~\cite{CMS,Sturm2} to be the set of midpoints of minimizing geodesics joining a given pair of sets.
		If instead one allows \emph{all} geodesics joining the pair of sets in the definition,
		then the inequality is of course still true, but is potentially much weaker---indeed,
		the resulting Minkowski average can then be very large (think of two sets on the torus).
		The magnetic setting provides a way of selecting a different subclass of geodesics when forming the Minkowski average---by adding a closed,
		cohomologically nontrivial one-form $\eta$ to the Riemannian Lagrangian.
		Closedness of $\eta$ means that magnetic geodesics are ordinary geodesics,
		and that the magnetic Ricci curvature coincides with the Riemannian Ricci curvature. However,
		the class of \emph{minimizing} magnetic geodesics may differ from the class of length-minimizing geodesics.
		Thus, for each cohomology class we get a potentially different Minkowski average, and if, say,
		the manifold is geodesically convex and has nonnegative Ricci curvature,
		then all these Minkowski averages satisfy the (undistorted) Brunn--Minkowski inequality. 

		The following is a simple example in which adding a closed one-form strictly decreases the volume of the Minkowski average of a given pair of sets,
		making the Brunn--Minkowski inequality tighter for this pair.
		We will denote the magnetic Minkowski average by $A_\lambda^{\mathrm{mag}}$ to distinguish it from the geodesic Minkowski average $A_\lambda^{\mathrm{geo}}$.

		Let $(M,g)$ be the cylinder $M = S^1\times \RR$ with the flat metric $g = d\theta^2 + dt^2$, and let 
		\[
			A_0 = \left(0,\tfrac12\pi\right)\times \left(0,1\right) \quad \text{and} \quad A_1 := \left(\pi,\tfrac32\pi\right)\times \left(0,1\right).
		\]

		Then the geodesic Minkowski $1/2$-average of $A_0$ and $A_1$ is
		\[
			A_{1/2}^{\text{geo}} = \left(\left(\tfrac12\pi,\pi\right)\times\left(0,1\right)\right)\cup\left(\left(\tfrac32\pi,2\pi\right)\times\left(0,1\right)\right),
		\]
		and if we set $\eta = \tfrac12d\theta$ then the magnetic Minkowski $1/2$-average is
		\[
			A_{1/2}^{\text{mag}} = \left(\tfrac12\pi,\pi\right)\times\left(0,1\right).
		\]
		Indeed, the magnetic potential $\eta$ strictly reduces the action of geodesics with a positive $\theta$ displacement,
		thus breaking the tie between pairs of minimizing geodesics with the same length.
		A similar example can be constructed on the torus; see Figure~\ref{fig:cylinder-torus}.

		\begin{figure}[ht]
		\centering
		% ── helper commands for torus boundary ───────────────────────────────────
		\newcommand{\tftbdry}{%
		  \draw[thick](0,0)rectangle(4,4);
		  % faint half-period grid
		  \draw[gray!25,very thin](2,0)--(2,4)(0,2)--(4,2);
		  % identification marks: left/right edges (theta2-direction, pointing up)
		  \fill[black!55](-0.10,1.90)--(0.10,1.90)--(0,2.10)--cycle;
		  \fill[black!55](3.90,1.90)--(4.10,1.90)--(4,2.10)--cycle;
		  % identification marks: bottom/top edges (theta1-direction, pointing right)
		  \fill[black!55](1.90,-0.10)--(1.90,0.10)--(2.10,0)--cycle;
		  \fill[black!55](1.90,3.90)--(1.90,4.10)--(2.10,4)--cycle;
		  % axis ticks and labels
		  \foreach \v/\lab in {0/{0},2/{\pi},4/{2\pi}}{ % chktex 1
		    \draw(\v,0)--(\v,-0.12);
		    \draw(0,\v)--(-0.12,\v);\node[left,font=\tiny]at(-0.12,\v){$\lab$};
		  }
		  \foreach \v/\lab in {0/{0},4/{2\pi}}{ % chktex 1
		    \node[below,font=\tiny]at(\v,-0.15){$\lab$};
		  }
		  \node[font=\scriptsize]at(2,-0.62){$\theta_1$};
		  \node[left,font=\scriptsize]at(-0.55,2){$\theta_2$};
		}
		\newcommand{\tftAB}{%
		  \fill[figblue!50]  (0.5,0.5)circle(0.318);
		  \fill[figorange!55](2.5,2.5)circle(0.318);
		  \node[font=\tiny]at(0.5,0.48){$A_0$};
		  \node[font=\tiny]at(2.5,2.48){$A_1$};
		}
		% ── two-column tabular: left col = geodesic, right col = magnetic ────────
		\begin{tabular}{c@{\hspace{1cm}}c}
		  % ── Row 1: cylinder ────────────────────────────────────────────────────
		  \begin{tikzpicture}[scale=1.1]
		    % A0 (blue), geodesic average (gold, two pieces), A1 (orange)
		    \fill[figblue!50]    (0,0) rectangle (1,1.5);
		    \fill[figgold!65]    (1,0) rectangle (2,1.5);
		    \fill[figorange!55]  (2,0) rectangle (3,1.5);
		    \fill[figgold!65]    (3,0) rectangle (4,1.5);
		    \draw[gray!60,thin] (1,0)--(1,1.5) (2,0)--(2,1.5) (3,0)--(3,1.5);
		    \draw[thick] (0,0) rectangle (4,1.5);
		    \fill[black!65] (-0.09,0.68)--(0.09,0.68)--(0,0.84)--cycle;
		    \fill[black!65] (3.91,0.68)--(4.09,0.68)--(4,0.84)--cycle;
		    \foreach \x/\lab in {0/{0},1/{},2/{},3/{},4/{2\pi}}{
		      \draw (\x,0)--(\x,-0.08);
		      \node[below,font=\scriptsize] at (\x,-0.10) {$\lab$};
		    }
		    \node[left,font=\scriptsize] at (0,0)   {$0$};
		    \node[left,font=\scriptsize] at (0,1.5) {$1$};
		    \node[left,font=\scriptsize] at (-0.15,0.75) {$t$};
		    \node[font=\scriptsize] at (2,-0.43) {$\theta$};
		    \node[font=\scriptsize] at (2,-0.84) {$\eta=0$};
		    \node[font=\small] at (0.5, 0.75) {$A_0$};
		    \node[font=\small] at (2.5, 0.75) {$A_1$};
		    \node[font=\scriptsize] at (1.5,0.75) {$A_{1/2}^{\text{geo}}$};
		    \node[font=\scriptsize] at (3.5,0.75) {$A_{1/2}^{\text{geo}}$};
		  \end{tikzpicture}
		  &
		  \begin{tikzpicture}[scale=1.1]
		    % A0 (blue), magnetic average (green, one piece), A1 (orange)
		    \fill[figblue!50]   (0,0) rectangle (1,1.5);
		    \fill[figgreen!55]  (1,0) rectangle (2,1.5);
		    \fill[figorange!55] (2,0) rectangle (3,1.5);
		    % [3,4] white; eta^sharp arrows there only
		    \foreach \yi in {0.375,1.125}{
		      \draw[->,>=stealth,black!40,thin](3.32,\yi)--(3.68,\yi);
		    }
		    \draw[gray!60,thin] (1,0)--(1,1.5) (2,0)--(2,1.5) (3,0)--(3,1.5);
		    \draw[thick] (0,0) rectangle (4,1.5);
		    \fill[black!65] (-0.09,0.68)--(0.09,0.68)--(0,0.84)--cycle;
		    \fill[black!65] (3.91,0.68)--(4.09,0.68)--(4,0.84)--cycle;
		    \foreach \x/\lab in {0/{0},1/{},2/{},3/{},4/{2\pi}}{
		      \draw (\x,0)--(\x,-0.08);
		      \node[below,font=\scriptsize] at (\x,-0.10) {$\lab$};
		    }
		    \node[left,font=\scriptsize] at (0,0)   {$0$};
		    \node[left,font=\scriptsize] at (0,1.5) {$1$};
		    \node[left,font=\scriptsize] at (-0.15,0.75) {$t$};
		    \node[font=\scriptsize] at (2,-0.43) {$\theta$};
		    \node[font=\scriptsize] at (2,-0.84) {$\eta=\tfrac12 d\theta$};
		    \node[font=\small] at (0.5, 0.75) {$A_0$};
		    \node[font=\small] at (2.5, 0.75) {$A_1$};
		    \node[font=\scriptsize] at (1.5,0.75) {$A_{1/2}^{\text{mag}}$};
		  \end{tikzpicture}
		  \\[0.4cm]
		  % ── Row 2: flat torus ──────────────────────────────────────────────────
		  \begin{tikzpicture}[scale=0.82]
		    \tftbdry
		    \fill[figgold!60](1.5,1.5)circle(0.318);
		    \fill[figgold!60](3.5,3.5)circle(0.318);
		    \fill[figgold!60](1.5,3.5)circle(0.318);
		    \fill[figgold!60](3.5,1.5)circle(0.318);
		    \tftAB
		    \node[below,font=\scriptsize]at(2,-0.85){$\eta=0$};
		  \end{tikzpicture}
		  &
		  % eta = (1/2)(dtheta1 + dtheta2), arrows NE, cluster at (1.5,1.5)
		  \begin{tikzpicture}[scale=0.82]
		    \tftbdry
		    \foreach \xi in {0.5,1.5,2.5,3.5}{
		      \foreach \yi in {0.5,1.5,2.5,3.5}{
		        \draw[->,>=stealth,black!22,thin](\xi-0.127,\yi-0.127)--(\xi+0.127,\yi+0.127);
		      }
		    }
		    \fill[figgreen!60](1.5,1.5)circle(0.318);
		    \tftAB
		    \node[below,font=\scriptsize]at(2,-0.85){$\eta=\tfrac12(d\theta_1+d\theta_2)$};
		  \end{tikzpicture}
		\end{tabular}
		\caption{Geodesic (left) and magnetic (right) Minkowski averages on the cylinder with $\eta = \tfrac12 d\theta$ (top) and the flat torus with $\eta = \tfrac12(d\theta_1+d\theta_2)$ (bottom). The vector field $\eta^\sharp$ is shown in gray.}
		\label{fig:cylinder-torus} % chktex 24
		\end{figure}
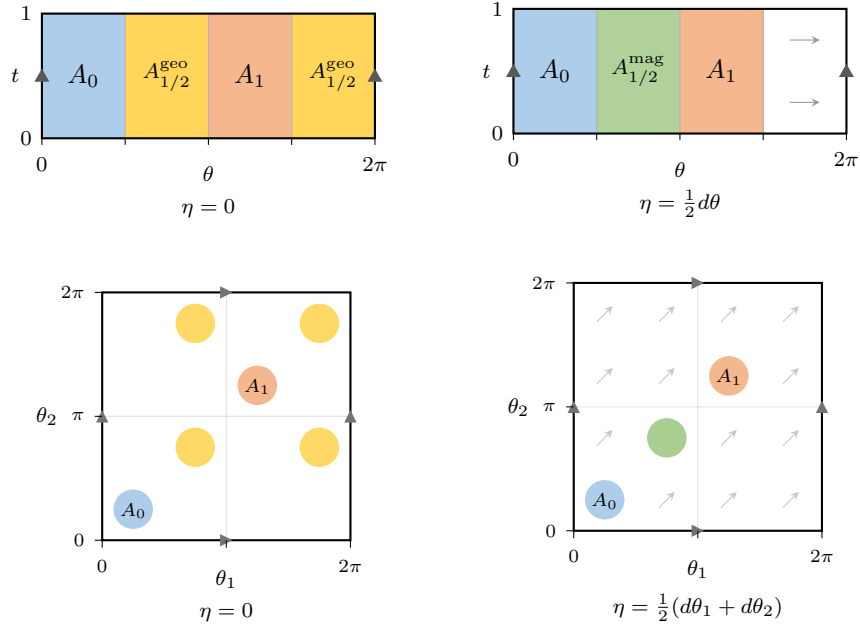

\section{From magnetic Ricci curvature to magnetic Brunn--Minkowski}\label{sec:ric-to-bm}

	Throughout this section and the next,
	we fix a smooth Riemannian manifold $(M,g)$ of dimension $n \ge 2$ and a one-form $\eta$ on $M$,
	such that conditions (I)-(III) from Section~\ref{sec:magnetic-geodesics} hold.
	We write $\la\cdot,\cdot\ra$ for the Riemannian inner product and $\nabla$ for the Levi-Civita connection.
	The musical isomorphisms $\sharp: T^*M \to TM$ and $\flat: TM \to T^*M$ are defined by $\sigma = \la\sigma^\sharp,\,\cdot\,\ra$ and $v^\flat = \la v,\,\cdot\,\ra$ for every $\sigma \in T^*M$ and $v \in TM$.

	In order to state the distorted magnetic Brunn--Minkowski inequality,
	we recall the comparison distortion coefficients from the synthetic curvature-dimension theory of Lott--Sturm--Villani~\cite{CMS,Sturm2,LV,Vil}.

	For $k \in \RR$, $t \in [0,1]$ and $\ell \ge 0$, define
	\[
		\tau_t^{k,n}(\ell) :=
		\begin{cases}
			t^{\tfrac1n}\left(\dfrac{\sin\left(t\ell\sqrt{\frac{k}{n-1}}\right)}{\sin\left(\ell\sqrt{\frac{k}{n-1}}\right)}\right)^{1 - \tfrac1n}, & k > 0 \text{ and } \ell\sqrt{\frac{k}{n-1}} < \pi, \\[4pt]
			+\infty, & k > 0 \text{ and } \ell\sqrt{\frac{k}{n-1}} \ge \pi, \\[4pt]
			t, & k = 0, \\[4pt]
			t^{\tfrac1n}\left(\dfrac{\sinh\left(t\ell\sqrt{\frac{-k}{n-1}}\right)}{\sinh\left(\ell\sqrt{\frac{-k}{n-1}}\right)}\right)^{1 - \tfrac1n}, & k < 0,
		\end{cases}
	\]
	with the convention $\tau_t^{k,n}(0)=t$, and for $A_0,A_1 \subseteq M$ let
	\begin{align*}
		\tau_t^{k,n}(A_0,A_1)
		:=
		\inf\left\{
			\tau_t^{k,n}(\ell)\,\middle|\,
			\begin{array}{c}
				\exists \text{ a minimizing magnetic geodesic } \gamma:[0,\ell]\to M\\
				\text{joining } A_0 \text{ to } A_1
			\end{array}
		\right\}.
	\end{align*}

	\begin{theorem}[Magnetic distorted Brunn--Minkowski inequality]\label{thm:distorted-bm}
		Suppose that $\Ric_\Omega \ge k$ on $SM$.
		Then for every pair $A_0,A_1 \subseteq M$ of Borel sets with positive volume and every $0 \le \lambda \le 1$,
		\begin{equation}\label{eq: distortedBM}
			\Vol(A_\lambda)^{1/n}
			\ge
			\tau_{1-\lambda}^{k,n}(A_0,A_1) \cdot \Vol(A_0)^{1/n}
			+
			\tau_{\lambda}^{k,n}(A_0,A_1) \cdot \Vol(A_1)^{1/n}.
		\end{equation}
	\end{theorem}

	The implication \textup{(i)}$\implies$\textup{(ii)} in Theorem~\ref{thm:main} follows immediately from Theorem~\ref{thm:distorted-bm} by setting $k = 0$.

	Theorem~\ref{thm:distorted-bm} is a special case of a general Brunn--Minkowski inequality for Tonelli Lagrangians satisfying a curvature-dimension condition, established in~\cite[Section~6]{Ass25}. The two key ingredients in the proof are a magnetic Bochner formula (Proposition~\ref{prop:magnetic-bochner}) and a needle decomposition result (Proposition~\ref{prop:magnetic-needle}). The latter will only be stated here, and the reader is referred to~\cite{Ass25} for a proof. 

	\begin{proposition}[Magnetic Bochner's formula]\label{prop:magnetic-bochner}
		Let $U \subseteq M$ be an open set and let $\sigma$ be a one-form on $U$ satisfying $d\sigma = \Omega$ and $|\sigma|\equiv 1$.
		Then
		\begin{equation}\label{eq:magnetic-bochner}
			-\la \sigma,d\delta\sigma\ra  + \left|\Sigma\right|^2 + \Ric_\Omega(\sigma^\sharp) = 0,
		\end{equation}
		where $\Sigma$ is the symmetric part of $\nabla\sigma\vert_{\ker\sigma}$. Moreover,
		\begin{equation}\label{eq:magnetic-bochnerineq}
			-\la \sigma,d\delta\sigma\ra  + \frac{(\delta\sigma)^2}{n-1} + \Ric_\Omega(\sigma^\sharp) \le 0,
		\end{equation}
		with equality if and only if $\Sigma$ is scalar, i.e.
		$\Sigma = -\tfrac{\delta\sigma}{n-1}\,g\vert_{\ker\sigma}$.
	\end{proposition}
	\begin{remark*} % chktex 1
		If $\Omega = 0$ then $\nabla\sigma$ is symmetric and vanishes on the span of $\sigma^\sharp$,
		so $|\Sigma|^2 = |\nabla \sigma|^2$ and we recover the usual Bochner's formula (in the case $|\sigma| \equiv 1$).
	\end{remark*}
	\begin{proof}
		By Bochner's formula for one-forms~\cite[Theorem 4.5.1]{Jost},
		\begin{equation}\label{eq: Bochner}
			-\Delta\frac{\la \sigma,\sigma\ra }{2} +  \la \Delta\sigma,\sigma\ra  = |\nabla\sigma|^2 + \Ric(\sigma^\sharp) ,
		\end{equation}
		where $\Delta = d\delta + \delta d$ is the Hodge Laplacian (which, for functions,
		differs by a sign from the Laplace-Beltrami operator), and $\nabla$ denotes covariant derivative.

		Since $|\sigma| \equiv 1$ and $d\sigma = \Omega$,
		\begin{equation}\label{eq:delta-alpha-alpha-1}
			\Delta\frac{\la \sigma,\sigma\ra }{2} = 0
			\qquad\text{and}\qquad
			\la \Delta\sigma,\sigma\ra  = \la d\delta\sigma,\sigma\ra  + \la \delta\Omega,\sigma\ra  = \la d\delta\sigma,\sigma\ra  + \delta\Omega(\sigma^\sharp).
		\end{equation}

		We now decompose the operator $\nabla \sigma$ into its tangential and normal components with respect to the direction $\sigma^\sharp$.
		In the tangential direction we have
		\[
			(\nabla\sigma)^*\sigma^\sharp  = \frac12 \cdot d\la \sigma,\sigma\ra  = 0,
		\]
		and, since $(\nabla\sigma) - (\nabla\sigma)^* = d\sigma = \Omega,$
		we also have
		\[
			(\nabla\sigma)\sigma^\sharp = (\nabla\sigma)^*\sigma^\sharp + \Omega(\sigma^\sharp,\cdot) = \Omega(\sigma^\sharp,\cdot) = \iota_{\sigma^\sharp}\Omega.
		\]

		Thus we can write
		\begin{equation}\label{eq:nabla-alpha-intermediate}
			|\nabla\sigma|^2 = |\iota_{\sigma^\sharp}\Omega|^2 + \left|\nabla\sigma\vert_{\ker\sigma}\right|^2.
		\end{equation}

		The normal component $\nabla\sigma\vert_{\ker\sigma}$ can be further decomposed into its symmetric and antisymmetric parts,
		which we denote by $\Sigma$ and $\Lambda$ respectively. The latter is given by
		\begin{align*}
			\Lambda & = \tfrac12 \left(\left(\nabla\sigma\vert_{\ker\sigma}\right) - \left(\nabla\sigma\vert_{\ker\sigma}\right)^*\right)
			\\
			& =
			\tfrac12((\nabla\sigma - \iota_{\sigma^\sharp}\Omega\otimes \sigma) - (\nabla\sigma - \iota_{\sigma^\sharp}\Omega\otimes \sigma)^*)
			\\
			& = \tfrac12\left(\Omega + (\iota_{\sigma^\sharp}\Omega)\wedge\sigma\right),
		\end{align*}
		whence
		\begin{equation}\label{eq:norm-A2}
			|\Lambda|^2 = \tfrac14\left(|\Omega|^2 - 2|\iota_{\sigma^\sharp}\Omega|^2\right).
		\end{equation}

		Since $\Sigma$ and $\Lambda$ are orthogonal in the Frobenius inner product,
		it follows from~\eqref{eq:nabla-alpha-intermediate} and~\eqref{eq:norm-A2} that
		\begin{equation}\label{eq:nabla-alpha-2}
			|\nabla\sigma|^2  = \tfrac12|\iota_{\sigma^\sharp}\Omega|^2 + \tfrac14|\Omega|^2 + |\Sigma|^2.
		\end{equation}
		Putting together~\eqref{eq: Bochner},~\eqref{eq:delta-alpha-alpha-1} and~\eqref{eq:nabla-alpha-2},
		we arrive at~\eqref{eq:magnetic-bochner}.
		To establish~\eqref{eq:magnetic-bochnerineq} we first note that since $(\nabla\sigma)\sigma^\sharp \perp \sigma^\sharp$,
		we have that
		\[\tr \, \Sigma = \tr(\nabla\sigma) = \div\sigma^\sharp = -\delta\sigma.\]
		Thus by the Cauchy-Schwarz inequality,
		\[
			|\Sigma|^2 \ge \frac{(\delta\sigma)^2}{n-1},
		\]
		with equality if and only if $\Sigma$ is scalar.
	\end{proof}
	\begin{lemma}\label{lem:tight}
		For every $x \in M$, every $v \in S_xM$ and every $c_0 \in \RR$ there exists a neighborhood $U \supseteq x$ and a one-form $\sigma$ on $U$ such that $d\sigma = \Omega$ and $|\sigma| \equiv 1$ on $U$, and such that
		\begin{equation}\label{eq:alpha-at-x}
			\sigma^\sharp\vert_x = v, \quad (\delta\sigma)(x) = c_0
			\qquad
			\text{ and }
			\qquad
			\la \sigma,d\delta\sigma\ra _x = \frac{c_0^2}{n-1} + \Ric_\Omega(v).
		\end{equation}
		If $c_0 = 0$, then the vector field $V : = \sigma^\sharp$ satisfies
		\[
			|V| \equiv 1, \qquad \la\nabla_VV,\,\cdot\,\ra = \Omega(V,\,\cdot\,),
		\]
		\[
			V\vert_x = v, \qquad (\div V)(x) = 0, \qquad
			(V\div V)(x) + \Ric_\Omega(v) = 0
		\]
		and
		\[
			2\cdot\la \nabla_wV,w'\ra  =  \Omega(w,w') + \Omega(v,w)\la v,w'\ra  + \Omega(v,w')\la v,w\ra , \qquad w,w' \in T_xM.
		\]
	\end{lemma}
	\begin{proof}
		By Proposition~\ref{prop:magnetic-bochner},
		we have equality in~\eqref{eq:magnetic-bochnerineq} when the symmetric part $\Sigma$ of the restriction of $\nabla\sigma\vert_x$ to $\ker\sigma$ is scalar. We can always find a one-form $\sigma$ defined on a small neighborhood of $x$ and satisfying $d\sigma = \Omega$, $|\sigma|\equiv 1$, $\sigma^\sharp\vert_x = v$, and $\Sigma = -\frac{c_0}{n-1}\,g\vert_{\ker\sigma}$; see~\cite[Lemma 2.2]{Ass25} and references therein. For such $\sigma$,~\eqref{eq:alpha-at-x} will hold. If $c_0 = 0$ then $\Sigma = 0$, and then the conditions $d\sigma = \Omega$ and $|\sigma|\equiv 1$ imply the desired properties of the vector field $V$.
	\end{proof}

	We now state a needle decomposition result for magnetic geodesics,
	which extends Klartag's needle decomposition result for Riemannian manifolds~\cite{Kl};
	see also~\cite{Oh18,CM}.

	\begin{proposition}\label{prop:magnetic-needle}
		Let $k \in \RR$, suppose that $\Ric_\Omega \ge k$ on $SM$, and let $f \in L^1(\Vol)$ satisfy
		\[
			\int_M f\,d\Vol = 0.
		\]
		Then there exist a measure space $(\sA,\nu)$ and a family of Borel measures $\{\mu_\alpha\}_{\alpha \in \sA}$ on $M$ such that:
		\begin{enumerate}[(i)]
			\item For $\nu$-almost every $\alpha \in \sA$, either $\mu_\alpha$ is a Dirac mass, or
			\begin{equation}\label{eq:needleform}
				\mu_\alpha = (\gamma_\alpha)_*(e^{-\psi_\alpha}dt),
			\end{equation}
			where $I_\alpha \subseteq \RR$ is an interval,
			$\gamma_\alpha:I_\alpha \to M$ is a minimizing magnetic geodesic,
			and $\psi_\alpha: I_\alpha \to \RR$ is a smooth function satisfying
			\begin{equation}\label{eq:magnetic-needle-cd}
				\ddot\psi_\alpha \ge k + \frac{\dot\psi_\alpha^2}{n-1}.
			\end{equation}
			\item For every Borel measurable function $h:M \to \RR$,
			the function $\alpha \mapsto \int_M h\,d\mu_\alpha$ is $\nu$-measurable and
			\begin{equation}\label{eq:magnetic-disintegration}
				\int_M h\,d\Vol = \int_{\sA}\left(\int_M h\,d\mu_\alpha\right)d\nu(\alpha).
			\end{equation}
			\item For $\nu$-almost every $\alpha \in \sA$,
			\begin{equation}\label{eq:magnetic-needle-mass-0}
				\int_M f\,d\mu_\alpha = 0.
			\end{equation}
			Moreover, if $\mu_\alpha$ is of the form~\eqref{eq:needleform}, then for every $t \in I_\alpha$,
			\begin{equation}\label{eq:magnetic-needle-mass}
				\int_M f\,d\mu_{\alpha,t} \ge 0,
				\qquad \text{where} \quad
				\mu_{\alpha,t} := (\gamma_\alpha)_*\left((e^{-\psi_\alpha}dt)\vert_{[t,\sup I_\alpha)}\right). % chktex 9
			\end{equation}
		\end{enumerate}
	\end{proposition}
	\begin{proof}
		This is precisely the needle decomposition theorem for Lagrangians~\cite[Theorem 5]{Ass25},
		applied to the magnetic Lagrangian $L(v)=\frac12|v|_g^2 + \tfrac12 -\eta(v)$,
		with $\mu=\Vol$ and $N=n$. Minimizing extremals for $L$ are minimizing magnetic geodesics,
		and the curvature-dimension condition $\CD(k,n)$ for the Lagrangian $L$ is equivalent to the lower bound $\Ric_\Omega \ge k$; this can be proved either by direct computation from the definition of the weighted Ricci curvature for Tonelli Lagrangians, as was done in~\cite[Section~7]{Ass25}, or from the Bochner inequality~\eqref{eq:magnetic-bochnerineq} combined with~\cite[Theorem 1]{Ass25}.
	\end{proof}

	\begin{proof}[Proof of Theorem~\ref{thm:distorted-bm}]
		By inner regularity of the measure $\Vol$, we may assume that $A_0$ and $A_1$ are compact. Set
		\[
			f := \frac{\chi_{A_1}}{\Vol(A_1)} - \frac{\chi_{A_0}}{\Vol(A_0)},
		\]
		where $\chi_{A_i}$ are the indicator functions of the set $A_i$,
		and apply Proposition~\ref{prop:magnetic-needle}.
		Fix $\alpha \in \sA$ such that $\mu_\alpha$ is not a Dirac mass, and write
		\[
			\mu_\alpha = (\gamma_\alpha)_*(\rho_\alpha\,dt), \qquad \rho_\alpha:=e^{-\psi_\alpha}.
		\]

		We now prove that the distorted Brunn--Minkowski inequality holds on the needle $\mu_\alpha$, i.e.
		\begin{equation}\label{eq:needlewiseBM}
			\mu_\alpha(A_\lambda)^{1/n}
			\ge
			\tau_{1-\lambda}^{k,n}(A_0,A_1)\mu_\alpha(A_0)^{1/n}
			+
			\tau_{\lambda}^{k,n}(A_0,A_1)\mu_\alpha(A_1)^{1/n}.
		\end{equation}

		By~\eqref{eq:magnetic-needle-cd},
		the measure $\rho_\alpha(t)\,dt$ satisfies the one-dimensional $\CD(k,n)$ condition~\cite{Sturm2,Vil}.
		Equivalently, the density $\rho_\alpha$ satisfies
		\begin{equation}\label{eq:rho-cd}
			\rho_\alpha((1-\lambda)t_0+\lambda t_1)^{1/(n-1)}
			\ge
			(1-\lambda)\big(\beta_{1-\lambda}^{k,n}(t_1-t_0)\rho_\alpha(t_0)\big)^{\frac1{n-1}}
			+
			\lambda\big(\beta_{\lambda}^{k,n}(t_1-t_0)\rho_\alpha(t_1)\big)^{\frac1{n-1}}
		\end{equation}
		for all $t_0\le t_1$ in $I_\alpha$, where
		\[
			\beta_t^{k,n} := t^{-n}\left(\tau_t^{k,n}\right)^n.
		\]

		Set
		\[
			A_{\alpha,0}:=\gamma_\alpha^{-1}(A_0), \qquad
			A_{\alpha,1}:=\gamma_\alpha^{-1}(A_1), \qquad
			A_{\alpha,\lambda}:=\gamma_\alpha^{-1}(A_\lambda)
		\]
		and consider on $I_\alpha$ the three functions
		\[
			h_{\alpha,0}:=\beta_{1-\lambda}^{k,n}(A_0,A_1)\rho_\alpha\chi_{A_{\alpha,0}}, \qquad
			h_{\alpha,1}:=\beta_{\lambda}^{k,n}(A_0,A_1)\rho_\alpha\chi_{A_{\alpha,1}}, \qquad
			h_{\alpha,\lambda}:=\rho_\alpha\chi_{A_{\alpha,\lambda}},
		\]
		where $\beta_t^{k,n}(A_0,A_1)$ is the infimum of $\beta_t^{k,n}(\ell)$ over all lengths $\ell$ of minimizing magnetic geodesics joining $A_0$ to $A_1$. A directed version of the one-dimensional Borell--Brascamp--Lieb inequality, see~\cite[Lemma~3.8]{AK}, asserts that if $h_0,h_1,h_\lambda$ are nonnegative integrable functions on an interval $I\subseteq\RR$ such that
		\[
			h_\lambda((1-\lambda)t_0+\lambda t_1)
			\ge
			\left(
				(1-\lambda)h_0(t_0)^{\frac{1}{n-1}}
				+
				\lambda h_1(t_1)^{\frac{1}{n-1}}
			\right)^{n-1}
		\]
		for every $t_0\le t_1$ with $h_0(t_0)h_1(t_1)>0$, and
		\[
			\frac{\int_{I\cap[t,\infty)}h_0}{\int_I h_0}
			\le
			\frac{\int_{I\cap[t,\infty)}h_1}{\int_I h_1}
			\qquad\text{for every } t\in I,
		\]
		then
		\begin{equation}\label{eq:1DBBL}
			\int_I h_\lambda
			\ge
			\left(
				(1-\lambda)\left(\int_I h_0\right)^{1/n}
				+
				\lambda\left(\int_I h_1\right)^{1/n}
			\right)^n.
		\end{equation}
		We verify these hypotheses for the functions $h_{\alpha,0}, h_{\alpha,1}, h_{\alpha,\lambda}$.

		First, let $t_0 \le t_1$. If both $h_{\alpha,0}(t_0)$ and $h_{\alpha,1}(t_1)$ are nonzero,
		then $\gamma_\alpha(t_0)\in A_0$ and $\gamma_\alpha(t_1)\in A_1$, so by definition of $A_\lambda$,
		\begin{equation}\label{eq:in-A-lambda}
			\gamma_\alpha((1-\lambda)t_0+\lambda t_1)\in A_\lambda.
		\end{equation}

		Moreover, the restriction of $\gamma_\alpha$ to $[t_0,t_1]$ is a minimizing magnetic geodesic joining $A_0$ to $A_1$,
		whence
		\begin{equation}\label{eq: betas}
			\beta_{1-\lambda}^{k,n}(A_0,A_1)\le \beta_{1-\lambda}^{k,n}(t_1-t_0),
			\qquad\text{and}\qquad
			\beta_{\lambda}^{k,n}(A_0,A_1)\le \beta_{\lambda}^{k,n}(t_1-t_0).
		\end{equation}

		By~\eqref{eq:in-A-lambda},
		we have that $h_{\alpha,\lambda}((1-\lambda)t_0+\lambda t_1) \ge \rho_\alpha((1-\lambda)t_0+\lambda t_1)$.
		Hence, by~\eqref{eq:rho-cd} and~\eqref{eq: betas},
		\[
			h_{\alpha,\lambda}((1-\lambda)t_0+\lambda t_1)
			\ge
			\left(
				(1-\lambda)h_{\alpha,0}(t_0)^{\frac1{n-1}}
				+
				\lambda h_{\alpha,1}(t_1)^{\frac1{n-1}}
			\right)^{n-1}
		\]
		for every $t_0\le t_1$.

		Second, by~\eqref{eq:magnetic-needle-mass},
		\[
			\frac{\mu_{\alpha,t}(A_0)}{\Vol(A_0)} \le \frac{\mu_{\alpha,t}(A_1)}{\Vol(A_1)}
			\qquad\text{for every } t\in I_\alpha,
		\]
		and by~\eqref{eq:magnetic-needle-mass-0},
		\begin{equation}\label{eq:mu-alpha-mb}
			\frac{\mu_\alpha(A_0)}{\Vol(A_0)} = \frac{\mu_\alpha(A_1)}{\Vol(A_1)}
		\end{equation}
		for $\nu$-almost every $\alpha \in \sA$. If $\mu_\alpha(A_0)$ and $\mu_\alpha(A_1)$ are both positive then
		\[
			\frac{\mu_{\alpha,t}(A_0)}{\mu_\alpha(A_0)}
			\le
			\frac{\mu_{\alpha,t}(A_1)}{\mu_\alpha(A_1)}
			\qquad\text{for every } t\in I_\alpha.
		\]
		Otherwise $\mu_\alpha(A_0)=\mu_\alpha(A_1)=0$ and the same inequality holds trivially. 
		Since
		\[
			\mu_{\alpha,t}(A_i)=\int_{I_\alpha\cap[t,\infty)}\rho_\alpha\chi_{A_{\alpha,i}}\,dt
			\qquad\text{and}\qquad
			\mu_\alpha(A_i)=\int_{I_\alpha}\rho_\alpha\chi_{A_{\alpha,i}}\,dt,
		\]
		and since each $h_{\alpha,i}$ differs from $\rho_\alpha\chi_{A_{\alpha,i}}$ by a positive multiplicative constant,
		it follows that indeed
		\[
			\frac{\int_{I_\alpha\cap[t,\infty)}h_{\alpha,0}}{\int_{I_\alpha}h_{\alpha,0}}
			\le
			\frac{\int_{I_\alpha\cap[t,\infty)}h_{\alpha,1}}{\int_{I_\alpha}h_{\alpha,1}}
			\qquad\text{for every } t\in I_\alpha.
		\]

		Thus the hypotheses of the one-dimensional directed Borell--Brascamp--Lieb inequality~\eqref{eq:1DBBL} are satisfied by the functions $h_{\alpha,0},h_{\alpha,1},h_{\alpha,\lambda}$ on the interval $I_\alpha$, and we get:
		\[
			\int_{I_\alpha} h_{\alpha,\lambda}\,dt
			\ge
			\left(
				(1-\lambda)\left(\int_{I_\alpha} h_{\alpha,0}\,dt\right)^{1/n}
				+ 
				\lambda\left(\int_{I_\alpha} h_{\alpha,1}\,dt\right)^{1/n}
			\right)^n
		\]
		
		By the definitions of $h_{\alpha,0},h_{\alpha,1},h_{\alpha,\lambda}$, and $\beta_t^{k,n}$,
		this is equivalent to~\eqref{eq:needlewiseBM}.

		If $\mu_\alpha$ is a Dirac mass then~\eqref{eq:needlewiseBM} is immediate since $\tau_t^{k,n}(0)=t$ and $\mu_\alpha(A_0)=\mu_\alpha(A_1)$.
		Combining~\eqref{eq:needlewiseBM} with~\eqref{eq:mu-alpha-mb}, we obtain
		\[
			\mu_\alpha(A_\lambda)
			\ge
			\frac{\mu_\alpha(A_0)}{\Vol(A_0)}
			\left(
				\tau_{1-\lambda}^{k,n}(A_0,A_1)\Vol(A_0)^{1/n}
				+
				\tau_{\lambda}^{k,n}(A_0,A_1)\Vol(A_1)^{1/n}
			\right)^n
		\]
		for $\nu$-almost every $\alpha \in \sA$.
		Integrating with respect to $\nu$ and using~\eqref{eq:magnetic-disintegration}, we get
		\begin{align*}
			\Vol(A_\lambda)
			&=
			\int_{\sA}\mu_\alpha(A_\lambda)\,d\nu(\alpha) \\
			&\ge
			\left(
				\tau_{1-\lambda}^{k,n}(A_0,A_1)\Vol(A_0)^{1/n}
				+
				\tau_{\lambda}^{k,n}(A_0,A_1)\Vol(A_1)^{1/n}
			\right)^n
			\frac{1}{\Vol(A_0)}
			\int_{\sA}\mu_\alpha(A_0)\,d\nu(\alpha) \\
			&=
			\left(
				\tau_{1-\lambda}^{k,n}(A_0,A_1)\Vol(A_0)^{1/n}
				+
				\tau_{\lambda}^{k,n}(A_0,A_1)\Vol(A_1)^{1/n}
			\right)^n,
		\end{align*}
		which is~\eqref{eq: distortedBM}.
	\end{proof}

\section{From magnetic Brunn--Minkowski to magnetic Ricci curvature}\label{sec:bm-to-ric}

	In this section we prove the converse to Theorem~\ref{thm:distorted-bm}:
	\begin{theorem}\label{thm:converse}
		Let $k \in \RR$. Assume that for every pair $A_0,A_1 \subseteq M$ of Borel sets with positive volume,
		their magnetic Minkowski average $A_{1/2}$ satisfies
		\begin{equation}\label{eq:converse-dbm}
			\Vol(A_{1/2})^{1/n} \ge \tau_{1/2}^{k,n}(A_0,A_1)\left(\Vol(A_0)^{1/n} + \Vol(A_1)^{1/n}\right).
		\end{equation}
		Then $\Ric_\Omega \ge k$.
	\end{theorem}
	The idea of the proof is to work in a small neighborhood in which the magnetic midpoint of every pair of points is unique and depends smoothly on the endpoints.
	We then compute the linearization of the magnetic midpoint map.
	We take the sets $A_0$ and $A_1$ to be exponentiated ellipsoids,
	chosen so that, after linearization,
	the magnetic Minkowski average of $A_0$ and $A_1$ becomes the Euclidean Minkowski average of two concentric balls,
	an equality case of the Euclidean Brunn--Minkowski inequality.

	We begin by stating the magnetic analogue of the Jacobi equation; see also~\cite{Gouda,Ada97,Gro}.
	Write $\Omega^\sharp$ for the $(1,1)$ tensor defined by
	\[
		g(\Omega^\sharp \,\cdot\,,\,\cdot\,) = \Omega.
	\] 
	Denote by $R$ the Riemann curvature tensor, and for $v \in TM$ set
	\[
		R_v^\Omega : = R(\,\cdot\,,v)v - \nabla\Omega^\sharp(\,\cdot\,,v).
	\]

	\begin{lemma}[Magnetic Jacobi equation]\label{lem:magnetic-jacobi}
		Let $\gamma_s:[0,\ell(s)] \to M$ be a family of (unit speed) magnetic geodesics depending smoothly on a parameter $s \in (-s_0,s_0)$.
		Then the vector field
		\[
			S(t) : = \left.\frac{\partial}{\partial s}\right|_{s=0}\gamma_s(t)
		\]
		along the magnetic geodesic $\gamma_0$ satisfies the linear ordinary differential equation
		\begin{equation}\label{eq:magnetic-jacobi}
			\nabla_{\dot\gamma}\nabla_{\dot\gamma}S - \Omega^\sharp\nabla_{\dot\gamma}S +  R_{\dot\gamma}^\Omega(S) = 0.
		\end{equation}
	\end{lemma}
	\begin{proof}
		Write $T(t,s) = \partial_t\gamma_s(t) = \dot\gamma_s(t)$.
		Since each $\gamma_s$ is a magnetic geodesic,
		\[
			\nabla_TT = \Omega^\sharp T.
		\]
		Hence
		\begin{align*}
			\nabla_T\nabla_TS 
			& = 
			\nabla_T\nabla_ST
			\\
			& = - R(S,T)T + \nabla_S\nabla_TT
			\\
			& = - R(S,T)T + \nabla_S(\Omega^\sharp T)
			\\
			& = - R(S,T)T + (\nabla_S\Omega^\sharp)T + \Omega^\sharp\nabla_ST
			\\
			& = - R(S,T)T + (\nabla\Omega^\sharp)(S,T) + \Omega^\sharp\nabla_TS,
		\end{align*}
		as desired.
	\end{proof}

	\begin{definition}
		A solution to~\eqref{eq:magnetic-jacobi} will be called a \emph{magnetic Jacobi field}. 
	\end{definition}

	We now proceed to prove Theorem~\ref{thm:converse}. Fix $x \in M$ and $v \in S_xM$.
	Our goal is to show that
	\[
		\Ric_\Omega(v)\ge k.
	\]
	By Lemma~\ref{lem:tight}, there exists an open set $U \ni x$ and a vector field $V$ on $U$ such that 
	\begin{align}
		&|V|\equiv 1, \quad \nabla_VV = \Omega^\sharp V, \quad V\vert_x = v, \quad (\div V)(x) = 0, \notag\\[6pt]
		&(V\div V)(x) = - \Ric_\Omega(v), \qquad \text{and} \notag\\[6pt]
		&\nabla V\vert_x = \frac12\,\Omega^\sharp\vert_x + \frac12\,(\Omega^\sharp v) \otimes v^\flat + \frac12\,v \otimes (\Omega^\sharp v)^\flat. \label{eq:nabla-V-at-x}
	\end{align}

	Let $\Phi_t$ denote the flow of the vector field $V$,
	and let $\gamma$ be the integral curve of $V$ satisfying $\gamma(0)=x$.
	Then $\gamma$ is a magnetic geodesic with 
	\[
		\dot\gamma(0)=V\vert_x=v.
	\]

	Let $0 < \eps < 1$, and let $E_0,E_1 \subseteq T_xM$ be bounded open neighborhoods of the origin whose precise choice (depending on $\eps$) will be given later;
	for now we postulate that their diameters satisfy
	\[
		\max\{\diam(E_0),\diam(E_1)\} = O(\eps^3).
	\]	
	Here and in the sequel, the implied constants are allowed to depend on anything but $\eps$. Let
	\[
		A_0 := \exp_x(E_0)
		\qquad \text{and} \qquad
		A_1 := \exp_{\gamma(\eps)}(d\Phi_\eps(E_1)).
	\]
	
	For $w_0,w_1 \in T_xM$, let $\mip_\eps(w_0,w_1)$ denote the midpoint of the minimizing magnetic geodesic joining $\exp_x(w_0)$ to $\exp_{\gamma(\eps)}(d\Phi_\eps w_1)$; see Figure~\ref{fig:m-eps}. By Remark~\ref{rem:local-convexity}, if the neighborhood $U$ is chosen sufficiently small then every pair of points in $U$ is joined by a unique minimizing magnetic geodesic, which is contained in $U$; in particular, for every $(w_0,w_1)\in E_0\times E_1$, the point $\mip_\eps(w_0,w_1)$ is uniquely defined and depends smoothly on $w_0,w_1$. Hence, by the definition of the magnetic Minkowski average,
	\[
		A_{1/2} = \mip_\eps(E_0 \times E_1).
	\]

	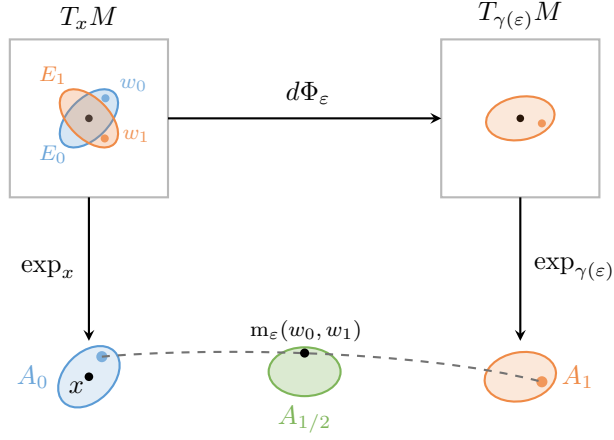
\begin{figure}
		\centering
		\begin{tikzpicture}[>=stealth, scale=0.95]

		% tracked points
		\coordinate (A0dot) at (0.18,0.28);
		\coordinate (A1dot) at (6.30,-0.07);
		\coordinate (w0dot) at (0.26,3.94);
		\coordinate (w1dot) at (0.22,3.32);
		\coordinate (dPhiw1dot) at (6.30,3.53);

		% --- colored sets drawn first so gamma overlays them ---
		\begin{scope}[shift={(0,0)}, rotate=45]
		\draw[figblue, thick, fill=figblue, fill opacity=0.18] (0,0) ellipse (0.50 and 0.34);
		\end{scope}
		\fill[figblue!80] (A0dot) circle (2.2pt);
		\node[figblue!90, font=\small] at (-0.78,0) {$A_0$};

		\begin{scope}[shift={(6,0)}, rotate=10]
		\draw[figorange!90, thick, fill=figorange, fill opacity=0.18] (0,0) ellipse (0.50 and 0.34);
		\end{scope}
		\fill[figorange!80] (A1dot) circle (2.2pt);
		\node[figorange!90, font=\small] at (6.8,0) {$A_1$};

		\draw[figgreen!90, thick, fill=figgreen!25] (3,0.07) ellipse (0.50 and 0.34);
		\node[figgreen!90, above=2pt, font=\small] at (3,-1) {$A_{1/2}$};

		% --- reference geodesic gamma, drawn on top of the sets ---
		\fill (0,0) circle (1.7pt);
		\node[below left=-2pt, font=\small] at (0,0) {$x$};

		% --- dashed sample magnetic geodesic ---
		% control points chosen so midpoint (3.0, 0.35) lies inside A_{1/2}
		\draw[dashed, black!55, thick]
		(A0dot) .. controls (1.5,0.40) and (4.5,0.40) .. (A1dot);
		\fill (3.00,0.33) circle (1.8pt);
		\node[above, font=\scriptsize] at (3.02,0.35) {$\mip_\eps(w_0,w_1)$};

		% === T_xM box (E_0 and E_1 both centered at the origin) ===
		\begin{scope}[shift={(0,3.60)}]
		\draw[gray!55, thick] (-1.10,-1.10) rectangle (1.10,1.10);
		\node[above, font=\small] at (0,1.10) {$T_xM$};
		\fill (0,0) circle (1.5pt);
		% E_0: origin-centered, 45 deg tilt, blue
		\begin{scope}[rotate=45]
		\draw[figblue, thick, fill=figblue, fill opacity=0.25] (0,0) ellipse (0.50 and 0.28);
		\end{scope}
		\node[figblue!90, font=\scriptsize] at (-.5,-.5) {$E_0$};
		\node[figblue!80, right=1pt, font=\scriptsize] at (0.25,0.47) {$w_0$};
		% E_1: origin-centered, -45 deg tilt, orange
		\begin{scope}[rotate=-45]
		\draw[figorange!90, thick, fill=figorange, fill opacity=0.25] (0,0) ellipse (0.50 and 0.28);
		\end{scope}
		\fill[figblue!80] (0.23,0.28) circle (1.6pt);
		\fill[figorange!80] (0.22,-0.28) circle (1.6pt);
		\node[figorange!90, font=\scriptsize] at (-.5,.6) {$E_1$};
		\node[figorange!80, right=1pt, font=\scriptsize] at (0.29,-0.25) {$w_1$};
		\end{scope}

		% --- exp_x arrow: from T_xM box down to A_0 ---
		\draw[->, black, thick] (0,2.50) -- (0,0.50);
		\node[left, black, font=\small] at (-0.05,1.50) {$\exp_x$};

		% === T_{gamma(eps)}M box ===
		\begin{scope}[shift={(6,3.60)}]
		\draw[gray!55, thick] (-1.10,-1.10) rectangle (1.10,1.10);
		\node[above, font=\small] at (0,1.10) {$T_{\gamma(\eps)}M$};
		\fill (0,0) circle (1.5pt);
		% d Phi_eps(E_1): origin-centered, slightly rotated
		\begin{scope}[rotate=10]
		\draw[figorange!90, thick, fill=figorange, fill opacity=0.18] (0,0) ellipse (0.48 and 0.30);
		\end{scope}
		\fill[figorange!80] (0.30,-0.07) circle (1.6pt);
		\end{scope}

		% --- dPhi_eps arrow: from T_xM to T_{gamma(eps)}M ---
		\draw[->, thick] (1.10,3.60) -- (4.90,3.60);
		\node[above, font=\small] at (3.05,3.66) {$d\Phi_\eps$};

		% --- exp_{gamma(eps)} arrow: from T_{gamma(eps)}M box down to A_1 ---
		\draw[->, black, thick] (6,2.50) -- (6,0.50);
		\node[right, black, font=\small] at (6.05,1.50) {$\exp_{\gamma(\eps)}$};
		\end{tikzpicture}
		\caption{The sets $A_0$ and $A_1$ and the map $\mip_\eps$. }
		\label{fig:m-eps}
	\end{figure}

	We now linearize the map $\mip_\eps$ using the magnetic Jacobi equation.
	\begin{lemma}\label{lem:midpointlinearization}
		For every $(w_0,w_1)\in E_0\times E_1$, we have that
		\begin{equation}\label{eq: midpoint-linearization}
			\mip_\eps(w_0,w_1)
			=
			\exp_{\gamma(\eps/2)}\left(\frac12\, L_{\eps,0}w_0
			+
			\frac12\, L_{\eps,1}w_1
			\right) + O(\eps^6),
		\end{equation}
		where
		\[
			L_{\eps,0},L_{\eps,1}:T_xM\to T_{\gamma(\eps/2)}M
		\]
		are linear maps satisfying
		\begin{align}
			\label{eq: detL0}
			\det(L_{\eps,0})
			&=
			1+\frac{\eps^2}{8} \cdot \Ric_\Omega(v)+O(\eps^3) \quad \text{and}
			\\
			\label{eq: detL1}
			\det(L_{\eps,1})
			&=
			1-\frac{3\eps^2}{8} \cdot \Ric_\Omega(v)+O(\eps^3).
		\end{align}
	\end{lemma}
	\begin{proof}

	Let $w_i \in E_i, \, \, i=0,1$. For $s \in (-1,1)$,
	let $\gamma_s:[0,\ell(s)]\to M$ be a unit-speed magnetic geodesic satisfying
	\begin{equation}\label{eq:gamma-boundary}
		\gamma_s(0)=\exp_x(s\cdot w_0)
		\qquad\text{and}\qquad
		\gamma_s(\ell(s))=\exp_{\gamma(\eps)}\left(s\cdot d\Phi_\eps w_1\right),
	\end{equation}
	and in particular $\gamma_0=\gamma$.
	Note that if $U$ is a sufficiently small neighborhood of $x$ and $\eps$ is sufficiently small,
	then the curve $\gamma_s$ is uniquely defined for all $s \in (-1,1)$ and depends smoothly on $s$,
	as does its length $\ell(s)$. By Lemma~\ref{lem:magnetic-jacobi}, the vector field
	\[
		S(t):=\left.\frac{\partial}{\partial s}\right|_{s=0}\gamma_s(t)
	\]
	along $\gamma$ is a magnetic Jacobi field.

	The definition of $\mip_\eps$ implies that $\gamma_s(\ell(s)/2) = \mip_\eps(s\cdot w_0,s\cdot w_1)$, whence
	\begin{align*}
		d\mip_\eps\vert_{(0,0)}(w_0,w_1)
		& =
		\left.\frac{d}{ds}\right|_{s=0}\left(\gamma_s(\ell(s)/2)\right)
		\\
		& = S(\eps/2) + \frac{\ell'(0)}{2}\cdot\dot\gamma(\eps/2),
	\end{align*}
	under the natural identification $T_0T_{\gamma(\eps/2)}M \cong T_{\gamma(\eps/2)}M$.
	Since $\diam(E_i) = O(\eps^3)$, we have $|w_i| = O(\eps^3)$, and so 
	\begin{equation}\label{eq: midpoint-expansion}
		\mip_\eps(w_0,w_1) = \exp_{\gamma(\eps/2)}\left(S(\eps/2) + \frac{\ell'(0)}{2}\dot\gamma(\eps/2)\right) + O(\eps^6).
	\end{equation}
	
	We now compute second order expansions of $S(\eps/2)$ and $\ell'(0)$ for small $\eps$. Write 
	\[
		T : = \dot\gamma.
	\]

	Since $\gamma$ is a magnetic geodesic and $S$ is a magnetic Jacobi field, $T$ and $S$ verify
	\begin{align}
		& \dot T = \Omega^\sharp T, & T(0) & = v, & \dot  T(0) & = \Omega^\sharp v, \label{eq: T-equation}
		\\
		& \ddot S - \Omega^\sharp\dot S+ R_{T}^\Omega(S)=0, & S(0) & = w_0, & 
		S(\eps) & = d\Phi_\eps w_1-\ell'(0)T(\eps), \label{eq: S-equation}
	\end{align}
	where dots denote covariant differentiation along $\gamma$.
	Here the boundary conditions on $S$ follow from~\eqref{eq:gamma-boundary} by differentiating with respect to $s$ at $s=0$ and $t=0,\eps$.
	Differentiating~\eqref{eq: S-equation} at $t=0$ we get
	\begin{equation}\label{eq:dddotS}
		\dddot S(0)
		=
		\left(\nabla_v\Omega^\sharp+(\Omega^\sharp)^2 - R_v^\Omega\right)\dot S(0)
		+
		O(|w_0|),
	\end{equation}
	where we have used $|S(0)| = |w_0| = O(\eps^3)$. 
	
	Let $P_t$ denote parallel translation along $\gamma$. By~\eqref{eq: S-equation} and~\eqref{eq:dddotS},
	\begin{equation}
	\begin{split}\label{eq:S-parallel-expansion}
		P_{-t}S(t)
		& =
		w_0
		+
		t\cdot \dot S(0)
		+
		\frac{t^2}{2}\cdot \left(\Omega^\sharp \dot S(0) - R_v^\Omega w_0\right)
		\\
		& \qquad
		+
		\frac{t^3}{6}\cdot \left(\nabla_v\Omega^\sharp + (\Omega^\sharp)^2 - R_v^\Omega\right)\dot S(0)
		+
		O(\eps^3|w_0|)+O(\eps^4(|w_0|+|\dot S(0)|)),
	\end{split}
	\end{equation}
	since the fourth derivative of $P_{-t}S(t)$ is uniformly bounded on $[0,\eps]$ by $C(|S(0)|+|\dot S(0)|)$.
	Moreover,
	\[
		|\dot S(0)|=O\left(\eps^{-1}|w_0-P_{-\eps}S(\eps)|\right)=O\left(\eps^{-1}(|w_0|+|S(\eps)|)\right).
	\]
	Setting $t = \eps$ and solving for $\dot S(0)$, we get
	\begin{equation}\label{eq:Sdot0-expansion}
	\begin{split}
		\dot S(0)
		& =
		\frac{1}{\eps}
		\cdot 
		\left(
			\mathrm{I}
			-
			\frac{\eps}{2}\cdot \Omega^\sharp
			-
			\frac{\eps^2}{6}\left(
				\nabla_v\Omega^\sharp
				-
				\frac12(\Omega^\sharp)^2
				-
				R_v^\Omega
			\right)
		\right)\left(u-w_0\right)
		\\
		&\qquad\qquad\qquad
		+
		\frac{\eps}{2}\cdot R_v^\Omega w_0
		+
		O(\eps^2(|w_0|+|u|)),
	\end{split}
	\end{equation}
	where we have set 
	\begin{equation}\label{eq: u-definition}
		u : = P_{-\eps}S(\eps).
	\end{equation}

	Substituting this into~\eqref{eq:S-parallel-expansion} with $t = \eps/2$ we find
	\begin{equation}\label{eq:S-half-expansion}
	\begin{aligned}
		P_{-\eps/2}S(\eps/2)
		& =
		\frac12\left(w_0+u\right)
		+
		\frac{\eps}{8}\cdot \Omega^\sharp\left(w_0-u\right)
		\\
		& \qquad +
		\frac{\eps^2}{16}\Big(
			(\nabla_v\Omega^\sharp)\left(w_0-u\right)
			+
			R_v^\Omega\left(w_0+u\right)
		\Big)
		+
		O(\eps^3(|w_0|+|u|)).
	\end{aligned}
	\end{equation}

	We now derive an asymptotic formula for the vector $u$.
	Since the flow $\Phi_t$ of the vector field $V$ is by unit-speed magnetic geodesics, the vector field
	\[
		J(t):=d\Phi_t w_1
	\]
	along $\gamma$ is a magnetic Jacobi field, and so it satisfies
	\begin{equation}\label{eq: J-initial-data}
		\ddot J - \Omega^\sharp \dot J + R_T^\Omega(J) = 0, \qquad J(0)=w_1,
		\quad
		\dot J(0)=\nabla_{w_1}V.
	\end{equation}

	Moreover, by~\eqref{eq: u-definition} and~\eqref{eq: S-equation},
	\begin{equation}\label{eq: u-from-JT}
		u = P_{-\eps}J(\eps)-\ell'(0)P_{-\eps}T(\eps).
	\end{equation}

	By~\eqref{eq: J-initial-data}, 
	\begin{equation*}
		P_{-\eps}J(\eps)
		=
		w_1
		+
		\eps\cdot\nabla_{w_1}V
		+
		\frac{\eps^2}{2}\cdot\left(
			\Omega^\sharp\nabla_{w_1}V
			-
			R_v^\Omega(w_1)
		\right)
		+
		O(\eps^3|w_1|),
	\end{equation*}
	and by~\eqref{eq: T-equation},
	\[
		P_{-\eps}T(\eps)
		=
		v
		+
		\eps\cdot\Omega^\sharp v
		+
		\frac{\eps^2}{2}\cdot\left(
			(\nabla_v\Omega^\sharp)v+(\Omega^\sharp)^2v
		\right)
		+
		O(\eps^3).
	\]
	Therefore, by~\eqref{eq: u-from-JT},
	\begin{equation}\label{eq: u-eps-expansion}
	\begin{split}
		u
		& =
		w_1-\ell'(0)v
		+
		\eps\bigl(\nabla_{w_1}V-\ell'(0)\Omega^\sharp v\bigr)
		\\
		& \qquad
		+
		\frac{\eps^2}{2}\Bigl(
			\Omega^\sharp\nabla_{w_1}V
			-
			R_v^\Omega(w_1)
			-
			\ell'(0)\bigl((\nabla_v\Omega^\sharp)v+(\Omega^\sharp)^2v\bigr)
		\Bigr)
		+
		O(\eps^3(|w_1|+|\ell'(0)|)).
	\end{split}
	\end{equation}

	Next, we derive an asymptotic formula for $\ell'(0)$. 	Since each curve $\gamma_s$ has unit speed,
	\begin{equation}\label{eq:ST-orthogonality}
		\la \dot S,T\ra = \left<\nabla_TS,T\right> = \left<\nabla_ST,T\right> = 0
		\qquad \text{ and similarly } \qquad
		\la \dot J,T\ra=0.
	\end{equation}

	Taking the inner product of~\eqref{eq:Sdot0-expansion} with $v$ and using~\eqref{eq:ST-orthogonality},
	skew-adjointness of $\Omega^\sharp$, and the identity
	$\la R_v^\Omega \,\cdot\,,v\ra=0$, we get
	\[
		0
		=
		\la u-w_0,v\ra
		+
		\frac{\eps}{2}\la u-w_0,\Omega^\sharp v\ra
		+
		O(\eps^2(|w_0|+|u|)).
	\]

	Substituting~\eqref{eq: u-eps-expansion}, and using $\la \nabla_{w_1}V,v\ra=0$,
	$\la \Omega^\sharp v,v\ra=0$ and $|v|=1$, we obtain
	\begin{equation}\label{eq: ellprime-first}
		\ell'(0)
		=
		\la w_1-w_0,v\ra
		+
		\frac{\eps}{2}\la w_1-w_0,\Omega^\sharp v\ra
		+
		O(\eps^2(|w_0|+|w_1|)).
	\end{equation}

	By~\eqref{eq: T-equation},
	\begin{equation}\label{eq: T-half-expansion}
		P_{-\eps/2}T(\eps/2)
		=
		v
		+
		\frac{\eps}{2}\Omega^\sharp v
		+
		\frac{\eps^2}{8}\left(
			(\nabla_v\Omega^\sharp)v+(\Omega^\sharp)^2v
		\right)
		+
		O(\eps^3).
	\end{equation}
	Thus,
	\begin{align*}
		P_{-\eps/2}\bigg(S(\eps/2)& +\frac{\ell'(0)}{2}T(\eps/2)\bigg)
		\\
		\qquad\stackrel{\eqref{eq:S-half-expansion},\,\eqref{eq: T-half-expansion}}{=}
		{}&
		\frac12\bigl(w_0+u+\ell'(0)v\bigr)
		+
		\frac{\eps}{8}\Bigl(
			\Omega^\sharp(w_0-u)
			+
			2\ell'(0)\Omega^\sharp v
		\Bigr)
		\\
		\qquad&
		+
		\frac{\eps^2}{16}\Bigl(
			(\nabla_v\Omega^\sharp)(w_0-u)
			+
			R_v^\Omega(w_0+u)
			+
			\ell'(0)\bigl((\nabla_v\Omega^\sharp)v+(\Omega^\sharp)^2v\bigr)
		\Bigr)
		\\
		\qquad&
		+
		O(\eps^3(|w_0|+|u|+|\ell'(0)|))
		\\
		\qquad
		\stackrel{\eqref{eq: u-eps-expansion}}{=}
		{}&
		\frac12(w_0+w_1)
		+
		\frac{\eps}{8}\Bigl(
			4\nabla_{w_1}V
			+
			\Omega^\sharp(w_0-w_1)
			-
			\ell'(0)\Omega^\sharp v
		\Bigr)
		\\
		\qquad
		&
		+
		\frac{\eps^2}{16}\Biggl(
			2\Omega^\sharp\nabla_{w_1}V
			+
			(\nabla_v\Omega^\sharp)(w_0-w_1)
			+
			R_v^\Omega(w_0-3w_1)
			\\
			& \qquad
			-
			\ell'(0)\Bigl(
				2(\nabla_v\Omega^\sharp)v
				+
				(\Omega^\sharp)^2v
				+
				R_v^\Omega v
			\Bigr)
		\Biggr)
		+
		O(\eps^3(|w_0|+|w_1|)).
	\end{align*}

	Using~\eqref{eq:nabla-V-at-x},~\eqref{eq: ellprime-first} and
	$R_v^\Omega v=-(\nabla_v\Omega^\sharp)v$, we obtain
	\begin{equation}\label{eq:midpoint-differential}
		P_{-\eps/2}\left(S(\eps/2) + \frac{\ell'(0)}{2} \cdot T(\eps/2)\right)
		=
		\frac12\widetilde L_{\eps,0}w_0
		+
		\frac12\widetilde L_{\eps,1}w_1,
	\end{equation}
	where 
	\begin{align*}
		\widetilde L_{\eps,0}
		& =
		I
		+
		\frac{\eps}{4}\left(
			\Omega^\sharp
			+
			\Omega^\sharp v\otimes v^\flat
		\right)
		\\
		& \qquad
		+
		\frac{\eps^2}{8}\left(
			R_v^\Omega
			+
			\nabla_v\Omega^\sharp
			+
			(\nabla_v\Omega^\sharp v)\otimes v^\flat
			+
			(\Omega^\sharp)^2v\otimes v^\flat
			+
			\Omega^\sharp v\otimes(\Omega^\sharp v)^\flat
		\right)
		+
		O(\eps^3),
		\\
		\widetilde L_{\eps,1}
		& =
		I
		+
		\frac{\eps}{4}\left(
			\Omega^\sharp
			+
			\Omega^\sharp v\otimes v^\flat
			+
			2\,v\otimes(\Omega^\sharp v)^\flat
		\right)
		\\
		& \qquad
		+
		\frac{\eps^2}{8}\left(
			-3\,R_v^\Omega
			-
			\nabla_v\Omega^\sharp
			-
			(\nabla_v\Omega^\sharp v)\otimes v^\flat
			+
			(\Omega^\sharp)^2
		\right)
		+
		O(\eps^3).
	\end{align*}

	A direct computation using $\Ric_\Omega(v) =  \tr(R_v^\Omega) + \tfrac12|\iota_v\Omega|^2 + \tfrac14|\Omega|^2$ yields
	\begin{equation}\label{eq: det-tildeL}
		\det(\widetilde L_{\eps,0})
		=
		1+\frac{\eps^2}{8}\Ric_\Omega(v)+O(\eps^3)
		\quad \text{ and } \quad 
		\det(\widetilde L_{\eps,1})
		=
		1-\frac{3\eps^2}{8}\Ric_\Omega(v)+O(\eps^3).
	\end{equation}

	The lemma follows from~\eqref{eq: midpoint-expansion},~\eqref{eq:midpoint-differential} and~\eqref{eq: det-tildeL} once we set $L_{\eps,i}:=P_{\eps/2}\widetilde L_{\eps,i}$.
	\end{proof}

	We now choose the sets $E_0,E_1$ by requiring that they have equal volumes,
	and that their respective images under $L_{\eps,0}$ and $L_{\eps,1}$ be balls centered at the origin in $T_{\gamma(\eps/2)}M$,
	with the ball $L_{\eps,0}(E_0)$ having radius $\eps^3$. See Figure~\ref{fig:E0E1}. In more detail: set
	\begin{equation}\label{eq: Ei-definition}
		E_0 : = L_{\eps,0}^{-1}(B_{\eps^3}(0)) 
		\qquad\text{and}\qquad
		E_1 : = L_{\eps,1}^{-1}(B_{\theta\eps^3}(0)),
	\end{equation}
	where $B_r(0)$ denotes a ball of radius $r$ centered at the origin in $T_{\gamma(\eps/2)}M$, and where
	\[
		\theta := \left(\frac{\det(L_{\eps,1})}{\det(L_{\eps,0})}\right)^{1/n}.
	\]

	Recall that so far we only required $\max\{\diam(E_0),\diam(E_1)\}=O(\eps^3)$;
	our choice is indeed consistent with this requirement,
	since $L_{\eps,i}$ are $O(\eps)$-close to the identity and since,
	by~\eqref{eq: detL0} and~\eqref{eq: detL1}, 
	\begin{equation}\label{eq: theta-expansion}
		\theta
		=
		1
		-
		\frac{\eps^2}{2n}\Ric_\Omega(v)
		+
		O(\eps^3).
	\end{equation}

	Write
	\[
		\rv : =\Vol(E_0) = \Vol(E_1).
	\]

	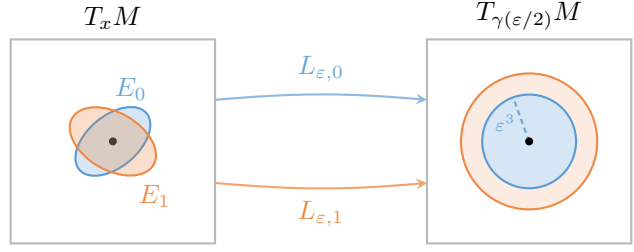
\begin{figure}
		\centering
		\begin{tikzpicture}[>=stealth, scale=1.0]

		% === T_xM (left box, E_0 and E_1 both centered at origin) ===
		\draw[gray!55, thick] (-1.35,-1.35) rectangle (1.35,1.35);
		\node[above, font=\small] at (0,1.35) {$T_xM$};
		\fill (0,0) circle (1.5pt);
		% \node[below left, font=\scriptsize] at (-0.02,-0.02) {$0$};

		% E_0: origin-centered, 40 deg tilt, blue
		\begin{scope}[rotate=40]
		\draw[figblue, thick, fill=figblue, fill opacity=0.25] (0,0) ellipse (0.58 and 0.34);
		\end{scope}
		\node[figblue!90, font=\small] at (0.24,0.7) {$E_0$};

		% E_1: origin-centered, -30 deg tilt, orange
		\begin{scope}[rotate=-30]
		\draw[figorange!90, thick, fill=figorange, fill opacity=0.25] (0,0) ellipse (0.62 and 0.38);
		\end{scope}
		\node[figorange!90, font=\small] at (0.55,-0.72) {$E_1$};

		% === T_{gamma(eps/2)}M (right box) ===
		\begin{scope}[shift={(5.5,0)}]
		\draw[gray!55, thick] (-1.35,-1.35) rectangle (1.35,1.35);
		\node[above, font=\small] at (0,1.35) {$T_{\gamma(\eps/2)}M$};

		% B_{theta eps3} (orange, larger) drawn first
		\draw[figorange!90, thick, fill=figorange!15] (0,0) circle (0.90);

		% B_eps3 (blue, smaller) on top
		\draw[figblue, thick, fill=figblue!25] (0,0) circle (0.62);

		% Radius indicator for the inner (blue) circle
		\draw[figblue!90, thick, dash pattern=on 2pt off 1.5pt] (0,0) -- (110:0.62);
		\node[figblue!90, font=\tiny, below left] at (95:0.5) {$\eps^3$};
		
		\fill (0,0) circle (1.5pt);

		\end{scope}

		% === Arrows L_{eps,0} and L_{eps,1} ===
		\draw[->, figblue!70, thick] (1.35, 0.55) to[out=5,in=175] (4.15, 0.55);
		\node[above, figblue!90, font=\small] at (2.75,0.65) {$L_{\eps,0}$};

		\draw[->, figorange!70, thick] (1.35,-0.55) to[out=-5,in=-175] (4.15,-0.55);
		\node[below, figorange!90, font=\small] at (2.75,-0.65) {$L_{\eps,1}$};

		\end{tikzpicture}
		\caption{The choice of $E_0$ and $E_1$.}
		\label{fig:E0E1}
	\end{figure}

	Recall that $(\div V)(x) = 0$ and $(V\div V)(x) = -\Ric_\Omega(v)$,
	and that $A_0 = \exp_x(E_0)$ and $A_1 = \exp_{\gamma(\eps)}(d\Phi_\eps(E_1))$ where $\Phi_t$ is the flow of $V$.
	Since $\diam(E_i)=O(\eps^3)$,
	it follows by standard asymptotic expansions of the Jacobian determinants of the exponential map and of the flow $\Phi_t$ that
	\[
		\Vol(A_0) = \rv\left(1 + O(\eps^3)\right)
		\qquad\text{and}\qquad
		\Vol(A_1) =
		\rv\cdot\left(1 - \frac{\eps^2}{2}\cdot\Ric_\Omega(v) + O(\eps^3)\right).
	\]

	Hence, by our assumption that the Brunn--Minkowski inequality~\eqref{eq:converse-dbm} holds,
	\begin{equation}\label{eq:brunn-minkowski-application}
	\begin{split}
		\Vol(A_{1/2})^{1/n}
		& \ge
		\tau_{1/2}^{k,n}(A_0,A_1)\cdot\left(\Vol(A_0)^{1/n}+\Vol(A_1)^{1/n}\right)
		\\
		& =
		\left(\frac12 + \eps^2\cdot\frac{k}{16n} + O(\eps^3)\right)\cdot \rv^{1/n} \cdot \left(2 - \frac{\eps^2}{2n} \cdot \Ric_\Omega(v)+ O(\eps^3)\right)
		\\
		& = 
		\rv^{1/n} \cdot \left(1 + \frac{\eps^2}{8n} \cdot \left(k - 2\Ric_\Omega(v)\right) + O(\eps^3)\right),
	\end{split}
	\end{equation}
	where in the second passage we used that every minimizing magnetic geodesic joining $A_0$ to $A_1$ has length $\eps + O(\eps^2)$,
	and expanded the distortion coefficient $\tau_{1/2}^{k,n}$ near $0$.
	
	We claim that
	\begin{equation}\label{eq: VolA12}
		\Vol(A_{1/2})
		=
		\Vol\left(B_{\frac{1 + \theta}{2}\eps^3}(0)\right)\left(1+O(\eps^3)\right).
	\end{equation}

	Indeed, write $r:=\frac{1+\theta}{2} \cdot \eps^3$ and define
	\[
		F_\eps(y)
		:=
		\exp_{\gamma(\eps/2)}^{-1}\!\left(
			\mip_\eps\left(
				L_{\eps,0}^{-1}\left(\frac{2}{1+\theta}\cdot y\right),
				L_{\eps,1}^{-1}\left(\frac{2\theta}{1+\theta} \cdot y\right)
			\right)
		\right),
		\qquad y\in B_r(0).
	\]

	By Lemma~\ref{lem:midpointlinearization}, there exists a constant $C > 0$ such that 
	\[
		\left| F_\eps(y) - y \right| \le C\cdot \eps^6, \qquad y \in B_r(0).
	\]

	In particular, for every $z\in B_{r-C\eps^6}(0)$ and every $y\in \partial B_r(0)$,
	\[
		|F_\eps(y)-y|\le C\eps^6 < r-|z|\le |y-z|,
	\]
	so $F_\eps(y)\neq z$ on $\partial B_r(0)$.
	Therefore $F_\eps$ is homotopic to the identity through maps avoiding $z$ on $\partial B_r(0)$ (take the linear interpolation $(1-t)\,F_\eps + t\,\mathrm{Id}$), and thus the map $F_\eps:B_r(0)\to\mathbb{R}^n$ has degree $1$ at $z$. In particular, $z\in F_\eps(B_r(0))$. Together with~\eqref{eq: midpoint-linearization}, this implies that
	\[
		B_{r-C\eps^6}(0)\subseteq \exp_{\gamma(\eps/2)}^{-1}(A_{1/2})
		\subseteq B_{r+C\eps^6}(0).
	\]

	Since $\eps^6=O(r\eps^3)$, estimate~\eqref{eq: VolA12} follows.

	By~\eqref{eq: VolA12},~\eqref{eq: theta-expansion},~\eqref{eq: Ei-definition} and~\eqref{eq: detL0},
	\begin{align*}
		\Vol(A_{1/2})^{1/n}
		& = 
		\left(\frac{1+\theta}{2}\right) \cdot \Vol(B_{\eps^3}(0))^{1/n}
		\\
		& =
		\left(
			1
			-
			\frac{\eps^2}{4n}\Ric_\Omega(v)
			+
			O(\eps^3)
		\right)\cdot\Vol(B_{\eps^3}(0))^{1/n}
		\\
		& = 
		\left(
			1
			-
			\frac{\eps^2}{4n}\Ric_\Omega(v)
			+
			O(\eps^3)
		\right)\cdot\det(L_{\eps,0})^{1/n}\cdot \rv ^{1/n}
		\\
		& = 
		\left(
			1
			-
			\frac{\eps^2}{4n}\Ric_\Omega(v)
			+
			O(\eps^3)
		\right)\cdot\left(1+\frac{\eps^2}{8n}\Ric_\Omega(v)+O(\eps^3)\right)\cdot \rv^{1/n}
		\\
		& =
		\left(
			1
			-
			\frac{\eps^2}{8n}\Ric_\Omega(v)
			+
			O(\eps^3)
		\right)\cdot \rv^{1/n}.
	\end{align*}
	Comparing with~\eqref{eq:brunn-minkowski-application} we see that $\Ric_\Omega(v) \ge k$,
	and thus the proof is complete.
	
\section{Contact magnetic Brunn--Minkowski inequality on the Heisenberg group}\label{sec:Heisenberg}

	\subsection{Overview}

	In this section we prove Theorem~\ref{thm:critical-heisenberg-bm}.
	The idea is to apply the distorted Brunn--Minkowski inequality, Theorem~\ref{thm:distorted-bm},
	to the magnetic Minkowski average corresponding to the magnetic potential $c\eta$ for $c<1$,
	since in this case the global conditions of the theorem are satisfied (see Lemma~\ref{lem:heisenberg-conditions}).
	We then let $c\nearrow 1$ and show that the distortion coefficients and Minkowski averages converge. 
	
	Some of the magnetic geodesics escape to infinity as $c \nearrow 1$,
	namely those joining pairs of points which are not connected by a minimizing magnetic geodesic when $c=1$.
	While posing a technical hurdle,
	this phenomenon has a desirable effect from the point of view of the Brunn--Minkowski inequality: it causes the magnetic Minkowski average to explode.
	Indeed, if a pair $(p_0,p_1) \in A_0\times A_1$ has no connecting minimizing contact magnetic geodesic,
	then the magnetic Minkowski average has infinite volume (provided that each $p_i$ is a density point of $A_i$); this is the content of Lemma~\ref{lem:P-infinite-midpoints}.
	
	\subsection{Minimizing magnetic geodesics in \texorpdfstring{$\Heis^m$}{the Heisenberg group}}

	We begin by giving a description of minimizing magnetic geodesics in $\Heis^m$.
	Inoguchi and Munteanu~\cite{IM26} recently described magnetic geodesics on $\Heis^1$ as orbits of one-parameter subgroups of the oscillator group.
	Our formulas will be more elementary, compare~\cite[Section 3.2]{DIMN}.
	
	We work in the Heisenberg group $\Heis^m$ ($m \ge 1$) with coordinates $(x,y,z) \in\RR^m\times\RR^m\times\RR$,
	in which the left-invariant Riemannian metric and contact form are given by
	\[
		g = \sum_{j=1}^m (dx^j)^2 + \sum_{j=1}^m (dy^j)^2 + \left(dz - \sum_{j=1}^m y^jdx^j\right)^2,\quad
		\eta := dz - \sum_{j=1}^m y^jdx^j.
	\]

	We denote the \emph{horizontal projection} by
	\[
		\Pi:\Heis^m \to \RR^{2m}, \qquad \Pi(x,y,z) := (x,y).
	\]

	On $\RR^{2m}$, we denote by $\bJ$ the standard complex structure and by $\left< \cdot,
	\cdot \right>_{\RR^{2m}}$, $|\cdot|_{\RR^{2m}}$ and $\Vol_{\RR^{2m}}$ the Euclidean inner product,
	norm and volume, respectively.
	A linear transformation of $\RR^{2m}$ will be called unitary if it is unitary under the identification $\RR^{2m}\cong\CC^m$.
	On $\RR^m$ we denote by $(e_j)_{j=1}^m$ the standard basis and by $\left<\cdot,\cdot\right>_{\RR^m}$ the standard inner product.
	A ball of radius $R$ centered at a point $p$ will be denoted by $B_R(p)$; sometimes we will also use a superscript to indicate the space,
	e.g. $B_R^{\RR^{2m}}(0)$.
	
	The group law in $\Heis^m$ is
	\[
		\Pi(pq)=\Pi p+\Pi q, \qquad
		z(pq)=z(p)+z(q)+\la y(p),x(q)\ra_{\RR^m}, \qquad p,q \in \Heis^m.
	\]
	
	For future reference, let us recall the explicit form of isometries of $\Heis^m$ fixing the origin and preserving the contact form.
	
	\begin{lemma}\label{lem:heisenberg-unitary-symmetries}
		Let $\rU:\RR^{2m}\to\RR^{2m}$ be unitary, and define
		\[
			f_{\rU}(v):=
			\tfrac12\la y(\rU v),x(\rU v)\ra_{\RR^m}
			-
			\tfrac12\la y(v),x(v)\ra_{\RR^m},
			\qquad v\in\RR^{2m}.
		\]

		Then the map
		\[
			\rF_{\rU}(v,z):=(\rU v,z+f_{\rU}(v))
		\]
		is an isometry of $(\Heis^m,g)$ preserving $\eta$.
	\end{lemma}

	\begin{proof}
		Write $\beta:=\sum_j y^j\,dx^j$. Since $\rU$ is unitary, we have that $df_{\rU}=\rU^*\beta-\beta$,
		whence
		\[
			\rF_{\rU}^*\eta
			=
			\rF_{\rU}^*(dz -\beta)
			=
			dz+df_{\rU}-\rU^*\beta
			=
			dz-\beta
			=
			\eta.
		\]

		Since $\rU$ is unitary, we also have $\rF_{\rU}^*g_{\RR^{2m}} = \rU^*g_{\RR^{2m}} = g_{\RR^{2m}}$,
		where $g_{\RR^{2m}}$ is the Euclidean metric on $\RR^{2m}$. Since $g = g_{\RR^{2m}} + \eta^2$,
		it follows that 
		$\rF_{\rU}^*g=g$.
	\end{proof}

	Let $c \in [0,1]$. We shall refer to unit-speed global minimizers of $\Len_g - c \cdot \int\eta$ as \emph{minimizing $c$-magnetic geodesics}.
	Only in this section, we parametrize magnetic geodesics by constant, rather than unit, speed,
	on the interval $[0,1]$.

	For $r > 0$ and $c \in [0,1)$, define
	\[
		\Psi_{c,r}(\theta) : =
		\begin{cases}
			\dfrac{(\theta-\sin\theta)r^2}{8\sin^2(\theta/2)}
			+
			\dfrac{\theta+c|\theta|\sqrt{1+\tfrac{(1-c^2)r^2}{4\sin^2(\theta/2)}}}{1-c^2}
			&
			\theta\in(-2\pi,2\pi)\setminus\{0\},
			\vspace{10pt}
			\\
			\dfrac{cr}{\sqrt{1-c^2}} 
			& 
			\theta =0.
		\end{cases}
	\]	
	For $r > 0$, set also
	\begin{equation}
		\Psi_{1,r}(\theta) : = 
		\frac{\theta}{2}-\frac{r^2}{4}\cot\frac{\theta}{2}, \qquad
		-2\pi<\theta<0.
	\end{equation}

	\begin{lemma}[Minimizing $c$-magnetic geodesics in $\Heis^m$]\label{lem:heisenberg-minimizers}
		Let $r \ge 0$ and let
		\[
			p_1 := (r e_1,0,z_1) \in \Heis^m.
		\]
		\begin{enumerate}[(i)]
			\item\label{item:heis-r-positive}
			Suppose that $r > 0$. Then for every $0 \le c \le 1$ there exists a unique minimizing $c$-magnetic geodesic $\gamma$ joining $0$ to $p_1$.
			The constant speed parametrization of $\gamma$ on the interval $[0,1]$ is given by $\gamma(t) = (\rx(t),\ry(t),\rz(t))$,
			where
			\begin{equation}\label{eq:normalized-minimizer-nonzero-xy}
				\rx(t)
				=
				R
				\left(\sin\left(\theta(t-\tfrac12)\right)+\sin\tfrac\theta2\right)e_1,
				\quad
				\ry(t)
				=
				R
				\left(\cos\left(\theta(t-\tfrac12)\right)-\cos\tfrac\theta2\right)e_1
			\end{equation}
			and
			\begin{equation}\label{eq:normalized-minimizer-nonzero-z}
				\rz(t)
				=
				tz_1
				+
				R^2\left[
					\tfrac14\left(\sin\left(2\theta(t-\tfrac12)\right)+(2t-1)\sin\theta\right)
					-
					\cos\tfrac\theta2\,
					\sin\left(\theta(t-\tfrac12)\right)
				\right].
			\end{equation}
			Here $\theta \in (-2\pi,2\pi)$ and $R \in \RR\setminus\{0\}$ are given by
			\[
				\theta=\Psi_{c,r}^{-1}(z_1)
				\qquad\text{and}\qquad
				R : = \frac{r}{2\sin(\theta/2)}.
			\]
			The case $\theta=0$ is understood by continuity; in that case
			\begin{equation}\label{eq:normalized-minimizer-line}
				\rx(t)=tr e_1,\qquad
				\ry(t)=0\qquad
				\text{and}
				\qquad
				\rz(t)=tz_1.
			\end{equation}
			In particular, $\Pi\gamma$ is a line segment when $\theta=0$,
			and a circular arc of radius $|R|$ and signed central angle $\theta$ otherwise. Moreover,
			\begin{equation}\label{eq:length-minimizer}
				\Len_g(\gamma) = 
				\sqrt{
					\frac{r^2\theta^2}{4\sin^2(\theta/2)}
					+
					\left(
						z_1-\frac{r^2(\theta-\sin\theta)}{8\sin^2(\theta/2)}
					\right)^2
				}
			\end{equation}
			where the right-hand side is understood by continuity to be $\sqrt{r^2+z_1^2}$ when $\theta=0$.

			\item\label{item:heis-r-zero-subcritical}
			Suppose that $r=0$ and $0\le c<1$.
			\begin{enumerate}[(a)]
				\item
					If $-\frac{2\pi}{1+c}\le z_1\le \frac{2\pi}{1-c}$ then the unique minimizing $c$-magnetic geodesic joining $0$ to $p_1$ is the vertical segment $(0,0,tz_1)$, and it has length $|z_1|$. 
				\item
					If $z_1>\frac{2\pi}{1-c}$ or $z_1<-\frac{2\pi}{1+c}$ then for every $\bJ$-invariant plane in $\RR^{2m}$ there exists a minimizing $c$-magnetic geodesic joining $0$ to $p_1$ whose horizontal projection is a full circle contained in that plane, and every minimizer is of this form.
			\end{enumerate}
			\item\label{item:heis-r-zero-critical}
			Suppose that $r=0$ and $c=1$.
			\begin{enumerate}[(a)]
				\item If $z_1>-\pi$ then the vertical segment $(0,0,tz_1)$ is the unique minimizing $c$-magnetic geodesic joining $0$ to $p_1$.
				\item If $z_1=-\pi$ then every positively oriented circle contained in a $\bJ$-invariant plane and enclosing nonnegative signed area occurs as the horizontal projection of a minimizing $c$-magnetic geodesic joining $0$ to $p_1$.
				\item If $z_1<-\pi$ then there is no minimizing $c$-magnetic geodesic joining $0$ to $p_1$.
			\end{enumerate}
		\end{enumerate}
	\end{lemma}

	\begin{proof}[Proof of Lemma~\ref{lem:heisenberg-minimizers}]
		Let $\gamma=(\rx,\ry,\rz)$ be a piecewise-$C^1$ curve joining $0$ to $p_1$, and
		write
		\[
			{l}:=\int |(\dot\rx,\dot\ry)|\,dt
			\qquad\text{and}\qquad
			\zeta:=\int \eta(\dot\gamma(t))\,dt.
		\]

		By Minkowski's inequality,
		\begin{equation*}
		\begin{split}
			\Len(\gamma)-c\int_\gamma\eta
			&=
			\int\sqrt{|(\dot\rx(t),\dot\ry(t))|^2+\eta(\dot\gamma(t))^2}\,dt
			-c\cdot\int\eta(\dot\gamma(t))\,dt
			\\
			&\ge
			\sqrt{{l}^2+\zeta^2}-c \cdot \zeta,
		\end{split}
		\end{equation*}
		with equality precisely when $|(\dot\rx,\dot\ry)|$ and $\eta(\dot\gamma)$ are
		proportional by a constant factor. Since the functional $\Len-c\int\eta$ is
		invariant under orientation-preserving reparametrization, we may assume that
		$\gamma$ is defined on $[0,1]$ and that
		$\Pi\gamma$ has constant (possibly zero) Euclidean speed. Thus it suffices to optimize over
		curves $\gamma = (\rx,\ry,\rz)$ satisfying $|(\dot\rx,\dot\ry)| \equiv {l}$ and
		$\eta(\dot\gamma) \equiv \zeta$. Since
		$\eta = dz - \sum_jy_jdx_j$ and $\rz(0) = 0$, this means that
		\begin{equation}\label{eq:rzt}
			\rz(t) = \int_0^t\la\ry(s),\dot\rx(s)\ra\,ds + \zeta \cdot t,
			\qquad t\in [0,1].
		\end{equation}

		Hence $\rz$ is determined uniquely by $(\rx,\ry)$, and it suffices to minimize the quantity
		\begin{equation}\label{eq:aux-heis-reduced-functional}
			\sqrt{{l}^2 + \zeta^2} - c \cdot \zeta, 
		\end{equation}
		over constant speed curves $(\rx,\ry):[0,1]\to\RR^{2m}$ joining the origin to $(re_1,0)$,
		where ${l}$ is the
		Euclidean arclength of the curve and
		\begin{equation}\label{eq:z1}
			\zeta = z_1 - \int_0^1\la\ry(t),\dot\rx(t)\ra\,dt.
		\end{equation}

		For fixed $\zeta$ a minimizer exists; it is a circular
		arc or a line segment, contained in a $\bJ$-invariant plane (this is known as Dido's problem).
		Thus we seek a minimizer of~\eqref{eq:aux-heis-reduced-functional} only among such curves.
		
		\begin{enumerate}[(i)]
			\item \textit{Suppose that $r > 0$.}
			By the above discussion, it remains to minimize~\eqref{eq:aux-heis-reduced-functional} over circular arcs or line segments $\alpha$ joining $(0,0)$ to $(r,0)$ in the Euclidean plane, where $l$ is the length of $\alpha$ and where
			\[
				\zeta = z_1-\int_\alpha y\,dx.
			\]

			We parametrize the family of such circular arcs $\alpha$ by their signed central angle $\theta \in (-2\pi,2\pi)$,
			where $\theta = 0$ corresponds to $\alpha$ being a line segment. The radius of $\alpha = \alpha(\theta)$ is
			\begin{equation}
				\frac{r}{2|\sin(\theta/2)|} \in (0,\infty],
			\end{equation}
			its arclength is
			\begin{equation}\label{eq:l}
				l= 
				\dfrac{r|\theta|}{2|\sin(\theta/2)|}
			\end{equation}
			and the signed area enclosed between it and the $x$ axis is
			\begin{equation}\label{eq:a}
				-\int_\alpha y\,dx
				=
				-\dfrac{r^2(\theta - \sin\theta)}{8\sin^2(\theta/2)},
			\end{equation}
			with~\eqref{eq:l} and~\eqref{eq:a} understood by continuity at $\theta=0$.
			Substituting~\eqref{eq:l} and~\eqref{eq:a} into~\eqref{eq:aux-heis-reduced-functional},
			we see that our goal is to minimize 
			\[
				f_{c,r}(\theta) : = 
				\begin{cases}
					\sqrt{\dfrac{r^2\theta^2}{4\sin^2(\theta/2)} + \left(z_1-\dfrac{r^2(\theta - \sin\theta)}{8\sin^2(\theta/2)}\right)^2} - c\cdot\left(z_1-\dfrac{r^2(\theta - \sin\theta)}{8\sin^2(\theta/2)}\right) &\theta \ne 0,
					\vspace{10pt}
					\\
					\sqrt{r^2 + z_1^2} - cz_1& \theta=0
				\end{cases}
			\]
			over the interval $(-2\pi,2\pi)$.
			The function $f_{c,r}$ is smooth and tends to $+\infty$ as $\theta\to\pm2\pi$ when $c<1$.
			When $c=1$, it tends to $+\infty$ as $\theta\to2\pi$ and to $2\pi$ as $\theta\to-2\pi$.
			
			For $c \in [0,1)$, the function $\Psi_{c,r}$ is smooth, strictly increasing,
			and tends to $\pm\infty$ as $\theta\to\pm2\pi$,
			while $\Psi_{1,r}$ is increasing from $-\infty$ to $+\infty$ on $(-2\pi,0)$. Moreover,
			a direct computation shows that $\mathrm{sgn}(f'_{c,r})=\mathrm{sgn}(\Psi_{c,r}-z_1)$.
			Therefore $\theta:=\Psi_{c,r}^{-1}(z_1)$ is the unique minimizer of $f_{c,r}$.
			The formulas in~\eqref{eq:normalized-minimizer-nonzero-xy},
			with the case $\theta=0$ understood by continuity to be~\eqref{eq:normalized-minimizer-line},
			then give a constant-speed parametrization of the minimizing arc.
			Finally,~\eqref{eq:rzt} and~\eqref{eq:z1} give~\eqref{eq:normalized-minimizer-nonzero-z} and the formula for $\rz$ in~\eqref{eq:normalized-minimizer-line}, and a direct computation gives the length formula~\eqref{eq:length-minimizer}.

			\item \textit{Suppose that $r=0$ and $0\le c<1$.}
			Since $r=0$, the curve $\Pi\gamma$ is either constant or a closed circle contained in a $\bJ$-invariant plane.
			Writing
			\[
				a:=-\int_\gamma \sum_j y_j\,dx_j,
			\]
			it therefore remains to minimize
			\[
				f_{c,0} : = \sqrt{4\pi |a|+(z_1+a)^2}-c(z_1+a)
			\]
			over such curves. The minimum is attained at $a=0$ precisely when
			\[
				-\frac{2\pi}{1+c}\le z_1\le \frac{2\pi}{1-c},
			\]
			in which case $\Pi\gamma$ is constant.
			Outside this interval the minimum is attained at a single nonzero value of $a$,
			positive when $z_1<-2\pi/(1+c)$ and negative when $z_1>2\pi/(1-c)$; the curve $\Pi\gamma$ can then be any closed circle contained in a $\bJ$-invariant plane and enclosing a signed area $a$.

			\item \textit{Suppose that $r=0$ and $c=1$.}
			The function $f_{1,0}$ is decreasing on $(-\infty,0]$,
			while on $[0,\infty)$ its derivative has the sign of $z_1+\pi$.
			Thus the vertical segment ($a=0$) is the unique minimizer when $z_1>-\pi$.
			When $z_1=-\pi$, the function $f_{1,0}$ is constant on $[0,\infty)$, so any positively oriented circle is a minimizer. If $z_1<-\pi$ then
			the infimum is $2\pi$, approached by positively oriented circles with radii tending to infinity and never attained.
		\end{enumerate}
		The proof of the lemma is complete.
	\end{proof}

	\begin{corollary}\label{cor:heisenberg-unique-minimizers}
		Let $0\le c\le1$ and let $p_0,p_1\in\Heis^m$. Then $p_0$ and $p_1$ are joined by a unique minimizing
		$c$-magnetic geodesic if and only if one of the following holds:
		\begin{enumerate}[(i)]
			\item $\Pi(p_0^{-1}p_1)\ne0$.
			\item $0\le c<1$, $\Pi(p_0^{-1}p_1)=0$, and
			\[
				-\frac{2\pi}{1+c}\le z(p_0^{-1}p_1)\le \frac{2\pi}{1-c}.
			\]
			\item $c=1$, $\Pi(p_0^{-1}p_1)=0$, and
			\[
				z(p_0^{-1}p_1)>-\pi.
			\]
		\end{enumerate}
	\end{corollary}

	\begin{proof}
		Since left translations preserve the metric and the contact form, it is enough to consider minimizing
		$c$-magnetic geodesics joining $0$ to $p:=p_0^{-1}p_1$. We may also
		apply the isometry $\rF_{\rU}$ from Lemma~\ref{lem:heisenberg-unitary-symmetries},
		where $\rU$ is a unitary transformation mapping
		$\Pi p$ to $(r e_1,0)$, with $r=|\Pi p|_{\RR^{2m}}$. If $r>0$,
		the result follows from
		Lemma~\ref{lem:heisenberg-minimizers}\ref{item:heis-r-positive}. If $r=0$, then $p=(0,0,z(p))$,
		and the isometries from Lemma~\ref{lem:heisenberg-unitary-symmetries} fix $p$ since $f_{\rU}(0)=0$.
		The remaining assertion follows directly from
		Lemma~\ref{lem:heisenberg-minimizers}\ref{item:heis-r-zero-subcritical},\ref{item:heis-r-zero-critical}.
	\end{proof}

	\subsection{Convergence of midpoints and distortion coefficients}
	Denote the center of $\Heis^m$ by 
	\[
		Z:=\{(0,0,z) \,\,\colon\,\,z \in \RR\},
	\]
	and set
	\[
		\cW:=\{(p_0,p_1)\in\Heis^m\times\Heis^m:
		\,\,p_0^{-1}p_1\in Z 
		\quad\text{and}\quad 
		z(p_0^{-1}p_1)\le-\pi\}.
	\]

	For $0\le c\le1$, whenever $p_0$ and $p_1$ are joined by a unique minimizing
	$c$-magnetic geodesic, denote its $\lambda$-midpoint by
	\[
		\mm_\lambda^c(p_0,p_1).
	\]

	Corollary~\ref{cor:heisenberg-unique-minimizers} implies 
	that the map $\mm_\lambda^1$ is defined on
	$(\Heis^m\times\Heis^m)\setminus\cW$, and that on any compact
	subset of $(\Heis^m\times\Heis^m)\setminus\cW$,
	the map $\mm_\lambda^c$ is defined for all $c$ sufficiently
	close to $1$.

	Recall from Section~\ref{sec:heisenberg-example} that the magnetic Ricci curvature associated to $\Omega_c = d(c\eta)$ is
	\[
		\Ric_{\Omega_c}
		=
		\frac m2(c-\eta)^2+\frac12(c^2-1)(1-\eta^2).
	\]

	In particular, for $0\le c\le 1$, the minimum of $\Ric_{\Omega_c}$ on each fiber of $SM$ equals
	\[
		-\frac{(m+1)(1-c^2)^2}{2(m+1-c^2)}.
	\]

	\begin{lemma}\label{lem:midpoint-convergence-away-from-critical-ray}
		Let $A_0,A_1\subseteq \Heis^m$ be compact sets with
		$(A_0\times A_1)\cap\cW = \varnothing$ and let
		$0<\lambda<1$. For $0\le c\le 1$, set
		\[
			k_c:=
			-\frac{(m+1)(1-c^2)^2}{2(m+1-c^2)}
			=
			\min_{S\Heis^m}\Ric_{\Omega_c}.
		\]
		Then, as $c\nearrow1$,
		\[
			\mm_\lambda^c\to\mm_\lambda^1
			\quad\text{uniformly on }A_0\times A_1
		\]
		and
		\[
			\tau_t^{k_c,n}(A_0,A_1)\longrightarrow t
			\qquad
			\text{for all }t\in(0,1),
		\]
		where
		\[
			n := \dim(\Heis^m) = 2m+1.
		\]
	\end{lemma}

	\begin{lemma}\label{lem:turning-angle-compactness}
		Let $c_j \in (0,1),r_j >0$ and $z_j \in \RR$, $j \ge 1$ be sequences satisfying
		\[
			c_j\nearrow1, \qquad 0 < r_j\to0
			\qquad
			\text{and}
			\qquad
			z_j\to z_\infty>-\pi.
		\]
		Then the sequence $\theta_j:=\Psi_{c_j,r_j}^{-1}(z_j)$ stays in a compact subset of $(-2\pi,2\pi)$.
	\end{lemma}

	\begin{proof}
		Passing to a convergent subsequence,
		we may assume that $\theta_j\to\theta\in[-2\pi,2\pi]$ and prove that $\theta\ne\pm2\pi$.
		Assume the contrary. Then after discarding finitely many terms,
		we may assume that $\theta_j\ne0$ for all $j$. Set
		\[
			\eps_j := 1-c_j^2
			\qquad\text{and}\qquad
			R_j:=\frac{r_j}{2\sin(\theta_j/2)}.
		\]
		Then
		\begin{equation}\label{eq:turning-angle-decomposition}
			z_j
			=
			\Psi_{c_j,r_j}(\theta_j)
			=
			\frac12 R_j^2(\theta_j - \sin\theta_j)
			+
			\frac{\theta_j + c_j|\theta_j|\sqrt{1 + \eps_j R_j^2}}{\eps_j}.
		\end{equation}

		If $\theta=2\pi$, then
		\[
			\frac{\theta_j\bigl(1+c_j\sqrt{1+\eps_j R_j^2}\bigr)}{\eps_j}
			\;\ge\;
			\frac{\theta_j}{\eps_j}
			\;\to\;+\infty
		\]
		while $\theta_j-\sin\theta_j\ge0$.
		Hence~\eqref{eq:turning-angle-decomposition} gives $z_j \to +\infty$, a contradiction.

		Assume next that $\theta=-2\pi$. Note that
		\begin{equation}\label{eq:turning-angle-small-parameter}
			\eps_j R_j^2 \to 0.
		\end{equation}

		Indeed, otherwise~\eqref{eq:turning-angle-decomposition} implies that the sequence $(z_j)$ has a subsequence tending to $-\infty$,
		a contradiction. 
		By~\eqref{eq:turning-angle-decomposition} and~\eqref{eq:turning-angle-small-parameter},
		\begin{align*}
			z_j
			&=
			\frac12 R_j^2(\theta_j-\sin\theta_j)
			+
			\theta_j\cdot
			\frac{1-\sqrt{1-\eps_j}\cdot\sqrt{1+\eps_j R_j^2}}{\eps_j}
			\\
			& =
			\frac{\theta_j}{2}
			-
			\frac12 R_j^2 \sin\theta_j
			+
			\frac{\theta_j \eps_j(1+R_j^2)^2}{8}(1+o(1)).
		\end{align*}
		Since $\theta_j\to-2\pi^+$, the last two terms are nonpositive for large $j$,
		whence $z_\infty = \lim_jz_j \le -\pi$, a contradiction.
	\end{proof}

	\begin{proof}[Proof of Lemma~\ref{lem:midpoint-convergence-away-from-critical-ray}]
		Recall that $\tau_t^{k,n}(A_0,A_1)$ is the infimum of
		$\tau_t^{k,n}$ over all lengths of minimizing
		magnetic geodesics joining $A_0$ to $A_1$, and that
		\[
			\tau_t^{k_c,n}(\ell)
			=
			t
			\left(
			\frac{\sinh(t\cdot h_c(\ell))}
			{t\sinh (h_c(\ell))}
			\right)^{\tfrac{n-1}{n}}
			\quad\text{where}\quad
			h_c(\ell)=\ell\sqrt{\frac{-k_c}{n-1}} = O(\ell\cdot(1-c)).
		\]

		Hence, by symmetry (i.e. by left-invariance and Lemma~\ref{lem:heisenberg-unitary-symmetries})
		it is enough to consider the minimizing
		$c$-magnetic geodesic $\gamma^c_{r,z}$ joining $0$ to $(r e_1,0,z)$ and to show
		that
		\[
			\gamma^c_{r,z}(\lambda)\to\gamma^1_{r,z}(\lambda) 
			\qquad\text{and}\qquad
			\Len(\gamma^c_{r,z}) = O(1) 
			\qquad\text{as }c\nearrow1,
		\]
		uniformly over $(r,z)$ in a given compact set
		\[
			K \subseteq \Big([0,\infty)\times\RR\Big)\setminus\Big(\{0\}\times(-\infty,-\pi]\Big).
		\]

		Fix such $K$. Since $K$ is compact,
		it suffices to consider an arbitrary sequence $(c_j,r_j,z_j) \in [0,1]\times K$ with $c_j\nearrow1$ and $(r_j,z_j) \to (r_\infty,z_\infty) \in K$,
		and, passing to a subsequence if we wish, prove that the sequence
		\[
			\gamma_j:=\gamma^{c_j}_{r_j,z_j}
		\]
		satisfies
		\begin{equation}\label{eq:gammajconvergence}
			\gamma_j(\lambda)
			\to
			\gamma^1_{r_\infty,z_\infty}(\lambda),
			\qquad\text{and}\qquad
			\Len(\gamma_j) = O(1).
		\end{equation}

		If $r_j=0$ along a subsequence then $r_\infty=0$,
		so since $(r_\infty,z_\infty)\in K$ we have $z_\infty>-\pi$,
		and Lemma~\ref{lem:heisenberg-minimizers} shows that along that subsequence $\gamma_j$ is the vertical segment from $0$ to $(0,0,z_j)$ for all $j$ sufficiently large. This immediately implies~\eqref{eq:gammajconvergence}. Thus we may assume that $r_j>0$ for all $j$.

		Let $\theta_j:=\Psi_{c_j,r_j}^{-1}(z_j)$.
		By~\eqref{eq:normalized-minimizer-nonzero-xy},~\eqref{eq:normalized-minimizer-nonzero-z},
		and~\eqref{eq:normalized-minimizer-line}, 
		\[
		\begin{split}
			\gamma_j(\lambda)
			&=
			\bigg(
				\frac{r_j\left(\sin(\theta_j(\lambda-\tfrac12))+\sin(\theta_j/2)\right)}{2\sin(\theta_j/2)}e_1,
				\frac{r_j\left(\cos(\theta_j(\lambda-\tfrac12))-\cos(\theta_j/2)\right)}{2\sin(\theta_j/2)}e_1,
				\\
				&\qquad
				\lambda z_j+
				\frac{r_j^2}{4\sin^2(\theta_j/2)}
				\left[
					\tfrac14\left(\sin(2\theta_j(\lambda-\tfrac12))+(2\lambda-1)\sin\theta_j\right)
					-\cos(\theta_j/2)\sin(\theta_j(\lambda-\tfrac12))
				\right]
			\bigg),
		\end{split}
		\]
		and by~\eqref{eq:length-minimizer},
		\[
			\Len(\gamma_j)
			=
			\sqrt{
				\frac{r_j^2\theta_j^2}{4\sin^2(\theta_j/2)}
				+
				\left(
					z_j-\frac{r_j^2(\theta_j-\sin\theta_j)}{8\sin^2(\theta_j/2)}
				\right)^2
			},
		\]
		with the case $\theta_j=0$ understood by continuity.
		If $r_\infty>0$ then $\theta_j\to\theta_\infty:=\Psi_{1,r_\infty}^{-1}(z_\infty)\in(-2\pi,0)$ whence $\gamma_j(\lambda) \to \gamma^1_{r_\infty,z_\infty}(\lambda)$, and if $r_\infty=0$ then Lemma~\ref{lem:turning-angle-compactness} shows that the sequence $\theta_j$ stays in a compact subset of $(-2\pi,2\pi)$ and then $\gamma_j(\lambda)\to(0,0,\lambda z_\infty)$. In both cases $\Len(\gamma_j)=O(1)$, since $\theta_j$ is bounded away from $\pm2\pi$.
		This finishes the proof of~\eqref{eq:gammajconvergence}.
	\end{proof}

\subsection{Proof of Theorem~\ref{thm:critical-heisenberg-bm}}
	Recall that Theorem~\ref{thm:critical-heisenberg-bm} asserts that for every pair of Borel sets $A_0,A_1\subseteq\Heis^m$ of positive volume
	and every $0\le\lambda\le1$, the contact magnetic Minkowski $\lambda$-average
	$A_\lambda$, defined to be the set of $\lambda$-midpoints of minimizing $1$-magnetic geodesics joining $A_0$ to $A_1$,
	satisfies
	\[
	\Vol(A_\lambda)^{1/n}
	\ge
	(1-\lambda)\cdot\Vol(A_0)^{1/n}
	+
	\lambda\cdot\Vol(A_1)^{1/n}.
	\]
	The cases $\lambda=0$ and $\lambda=1$ are trivial; in what follows we assume that $0<\lambda<1$.

	Before finishing the proof of the theorem we need one last lemma.

	\begin{lemma}\label{lem:P-infinite-midpoints}
	Suppose that there exist $(p_0,p_1)\in\cW$ such that each $p_i$ is a Lebesgue density point of the set $A_i$.
	Then $\Vol(A_\lambda) = \infty$.
	\end{lemma}

	\begin{proof}
	Recall that $(p_0,p_1) \in \cW$ means that $p_0^{-1}p_1 \in Z$ and $z(p_0^{-1}p_1) \le -\pi$.
	If $z(p_0^{-1}p_1)=-\pi$, then density and Fubini's theorem give
	another pair $(q_0,q_1)\in A_0\times A_1$ such that each $q_i$ is a Lebesgue
	density point of $A_i$, and such that $q_0^{-1}q_1\in Z$ and
	$z(q_0^{-1}q_1)<-\pi$. Thus we may assume that
	$z(p_0^{-1}p_1)<-\pi$. Furthermore, by symmetry, we may assume that $p_0=0$ and
	$p_1=(0,0,z_1)$ with $z_1<-\pi$. 

	By density and continuity of right translation, there exist $\delta>0$, a neighborhood
	$U$ of $p_1$, and a constant $a>0$ such that for every $p\in U$,
	\begin{equation}\label{eq:chunk-positive-measure}
	\Vol(Q_p)\ge a,
	\qquad\text{where}\quad
	Q_p:=B_\delta(0)\cap A_0\cap A_1p^{-1}.
	\end{equation}

	For $p\in U\setminus Z$, define
	\[
	C_p:=Q_p\,\mm_\lambda^1(0,p).
	\]

	Since $\Heis^m$ is unimodular, right translations preserve volume so by~\eqref{eq:chunk-positive-measure},
	\[
	\Vol(C_p)\ge a.
	\]

	If $q \in Q_p$ then $q \in A_0$ and $qp \in A_1$,
	and then by left invariance $q\mm_\lambda^1(0,p) = \mm_\lambda^1(q,qp) \in A_\lambda$. Hence
	\[
	C_p\subseteq A_\lambda.
	\]

	We now construct a sequence $p_j \in U\setminus Z$ such that the sets $C_{p_j}$ are pairwise disjoint.
	Let $r_j\searrow 0$ and set $p_j:=(r_je_1,0,z_1)$. Then $p_j\in U$ for $j$ sufficiently large.

	Let $\theta_j:=\Psi_{1,r_j}^{-1}(z_1)\in(-2\pi,0)$. Since 
	\[
	-\pi > z_1 = \frac{\theta_j}{2}-\frac{r_j^2}{4}\cot\frac{\theta_j}{2}
	\]
	and $r_j\searrow 0$, necessarily
	$r_j\cot(\theta_j/2)\to+\infty$ and in particular $\theta_j\to-2\pi$. 
	Hence, by Lemma~\ref{lem:heisenberg-minimizers}\ref{item:heis-r-positive},
	\begin{equation}\label{eq:Pimlambda10pj}
	|\Pi(\mm_\lambda^1(0,p_j))|_{\RR^{2m}}
	=
	\frac{r_j}{\sin(-\theta_j/2)}
	\left|\sin\frac{\lambda\theta_j}{2}\right| \to \infty.
	\end{equation}

	Moreover, since $Q_{p_j}\subseteq B^{\RR^{2m}}_\delta(0)$, 
	\begin{equation}\label{eq:PiCpj}
	\Pi(C_{p_j})
	=
	\Pi(Q_{p_j}\mm_\lambda^1(0,p_j))
	= 
	\Pi(Q_{p_j}) + \Pi\mm_\lambda^1(0,p_j)
	\subseteq
	B^{\RR^{2m}}_\delta(\Pi(\mm_\lambda^1(0,p_j))).
	\end{equation}

	By~\eqref{eq:Pimlambda10pj} and~\eqref{eq:PiCpj}, after passing to a subsequence,
	the sets $C_{p_j}$ are pairwise disjoint.
	Since each $C_{p_j}$ is a subset of $A_\lambda$ of volume $\ge a > 0$, it follows that
	$\Vol(A_\lambda)=\infty$.
	\end{proof}

	\begin{proof}[Proof of Theorem~\ref{thm:critical-heisenberg-bm}]
	We may assume that every point of $A_0$ and $A_1$ is a Lebesgue density point. Indeed,
	removing a null set from $A_0$ and $A_1$ does not change their volumes and can only decrease the volume of $A_\lambda$.

	If $(A_0\times A_1)\cap\cW\ne\varnothing$, then
	Lemma~\ref{lem:P-infinite-midpoints} gives
	$\Vol(A_\lambda)=\infty$, and there is nothing to prove. Thus we may assume that
	\[
	(A_0\times A_1)\cap\cW=\varnothing.
	\]

	Let $K_i\subseteq A_i$ be compact sets of positive volume. For
	$0<c\le1$, denote by $K_{\lambda}^{[c]}$ the set of $\lambda$-midpoints of minimizing $c$-magnetic geodesics joining $K_0$ to $K_1$.
	By Theorem~\ref{thm:distorted-bm} and Lemma~\ref{lem:heisenberg-conditions}, for all $c\in(0,1)$,
	\begin{equation}\label{eq:critical-bm-compact}
	\Vol(K_{\lambda}^{[c]})^{1/n}
	\ge
	\tau_{1-\lambda}^{k_c,n}(K_0,K_1)\Vol(K_0)^{1/n}
	+
	\tau_{\lambda}^{k_c,n}(K_0,K_1)\Vol(K_1)^{1/n},
	\end{equation}
	where $k_c:=\min_{S\Heis^m}\Ric_{\Omega_c}$. Since
	$K_0\times K_1$ is disjoint from $\cW$,
	Lemma~\ref{lem:midpoint-convergence-away-from-critical-ray} gives
	\begin{equation}\label{eq:critical-distortion-limit}
	\tau_{1-\lambda}^{k_c,n}(K_0,K_1)\to1-\lambda
	\qquad\text{and}\qquad
	\tau_{\lambda}^{k_c,n}(K_0,K_1)\to\lambda,
	\end{equation}
	as well as uniform convergence of $\mm_\lambda^c$ to $\mm_\lambda^1$ on
	$K_0\times K_1$, and therefore Hausdorff convergence of $K_{\lambda}^{[c]}$ to $K_{\lambda}^{[1]}$.
	By Hausdorff upper semicontinuity of volume on compact sets,
	it follows from~\eqref{eq:critical-bm-compact} and~\eqref{eq:critical-distortion-limit}
	that
	\[
	\Vol(K_{\lambda}^{[1]})^{1/n}
	\ge
	(1-\lambda)\Vol(K_0)^{1/n}
	+
	\lambda\Vol(K_1)^{1/n}.
	\]

	Since $K_i \subseteq A_i$, we have $K_{\lambda}^{[1]}\subseteq A_\lambda$, whence
	\[
	\Vol(A_{\lambda})^{1/n}
	\ge
	(1-\lambda)\Vol(K_0)^{1/n}
	+
	\lambda\Vol(K_1)^{1/n}.
	\]

	But $K_i$ are arbitrary compact subsets of $A_i$,
	so inner regularity of volume gives~\eqref{eq:critical-heisenberg-bm}.
	\end{proof}

	The following lemma shows that Theorem~\ref{thm:critical-heisenberg-bm} is sharp:
	equality is approached by pairs of 
	cylinders which differ by a vertical translation and
	whose radius is small compared to their height.

	\begin{lemma}\label{lem:heisenberg-cylinder-horizontal-midpoints}
	Let $0<\varrho\le1$ and $t,b>0$ satisfy
	\begin{equation}\label{eq:tbrho}
		t\ge b + 3\varrho^2,
	\end{equation}
	and define
	\[
	A_0:=B_\varrho\times[0,b]
	\qquad\text{and}\qquad
	A_1:=B_\varrho\times[t,t+b]
	\]
	where $B_\varrho:=B^{\RR^{2m}}_\varrho(0)$.
	Then the contact magnetic Minkowski $(1/2)$-average of $A_0$ and $A_1$ satisfies
	\begin{equation}\label{eq:cylinder-volume-defect}
	\Vol(A_{1/2})^{1/n}
	\le
	\left(1+O(\varrho^2/b)\right)
	\frac{\Vol(A_0)^{1/n}+\Vol(A_1)^{1/n}}2,
	\end{equation}
	where $n:=2m+1$.
	\end{lemma}

	\begin{proof}
	We will show that
	\begin{equation}\label{eq:cylinder-z-defect}
		A_{1/2}
		\subseteq
		B_\varrho\times
		\left[\frac t2-\frac\pi2\varrho^2,\,b+\frac t2+\frac\pi2\varrho^2\right],
	\end{equation}
	which will readily establish~\eqref{eq:cylinder-volume-defect}.
	Let $p_0\in A_0$ and $p_1\in A_1$ and set
	$r:=|\Pi p_1-\Pi p_0|_{\RR^{2m}}$. If $r=0$ then
	\[
		z(p_0^{-1}p_1)=z(p_1)-z(p_0)\ge t-b>0.
	\]

	Thus, by Lemma~\ref{lem:heisenberg-minimizers} and left-invariance,
	the midpoint of the minimizing magnetic geodesic joining $p_0$ to $p_1$ is $\mm_{1/2}^1(p_0,p_1) = (x(p_0),y(p_0),\tfrac12 (z(p_0)+z(p_1)))$ and therefore lies in the right-hand
	side of~\eqref{eq:cylinder-z-defect}. 

	Assume therefore that
	$r>0$. Set $p:=p_0^{-1}p_1$, let $\rU$ be a unitary map of $\RR^{2m}$ such that $\rU(\Pi p)=(re_1,0)$,
	and let $f_{\rU}$ and $\rF_{\rU}$ be as in Lemma~\ref{lem:heisenberg-unitary-symmetries}. Then $\rF_{\rU}(p)=(re_1,0,w)$ where
	\begin{equation}\label{eq:cylinder-w-lower-bound}
	\begin{split}
		w 
		& :=
		z(p)-\frac12\la y(p),x(p)\ra_{\RR^m}
		\\
		& =
		z(p_1)-z(p_0)-\la y(p_0),x(p_1)-x(p_0)\ra_{\RR^m}
		-\frac12\la y(p_1)-y(p_0),x(p_1)-x(p_0)\ra_{\RR^m}
		\\
		& \ge
		t-b-2\varrho^2
		\\
		& \ge
		\varrho^2.
	\end{split}
	\end{equation}
	by the Cauchy-Schwarz inequality, since $p_0,p_1 \in B_\varrho$, and by \eqref{eq:tbrho}.

	Let $\gamma$ be the minimizing $1$-magnetic geodesic joining $p_0$ to $p_1$, parametrized by constant speed on $[0,1]$. Then 
	\[
		\mm_{1/2}^1(p_0,p_1)=\gamma(1/2).
	\]
	
	Since left translations and $\rF_{\rU}$ are isometries preserving $\eta$ (Lemma~\ref{lem:heisenberg-unitary-symmetries}),
	the curve $\rF_{\rU}(p_0^{-1}\gamma)$ is the minimizing $1$-magnetic geodesic joining $\rF_{\rU}(0)=0$ to $\rF_{\rU}(p)=(re_1,0,w)$,
	and we may apply Lemma~\ref{lem:heisenberg-minimizers}\ref{item:heis-r-positive} to it.
	In particular,
	its horizontal projection is a circular arc of central angle $\theta \in (-2\pi,0)$ satisfying
	\begin{equation}\label{eq:w-theta-cylinder}
		w
		=
		\Psi_{1,r}(\theta)
		=
		\frac\theta2-\frac{r^2}{4}\cot\frac\theta2
		\le
		-\frac{r^2}{4}\cot\frac\theta2
	\end{equation}
	and radius
	\[
		R
		=
		\frac{r}{2\sin(-\theta/2)}
		\ge
		\frac{r\cot(-\theta/2)}{2}
		\ge
		\frac{2w}{r}
		\ge
		\varrho,
	\]
	where we used~\eqref{eq:cylinder-w-lower-bound} and $r\le2\varrho$. 
	Since $w > 0$, it follows from~\eqref{eq:w-theta-cylinder} that $\theta\in(-\pi,0)$.
	Since $\Pi\gamma$ differs from $\Pi\rF_{\rU}(p_0^{-1}\gamma)$ by a translation and a unitary transformation,
	it is also a circular arc with central angle $\theta\in(-\pi,0)$ and radius $R\ge\varrho$,
	and with endpoints $\Pi p_0,\Pi p_1\in B_\varrho$,
	and is therefore contained entirely in $B_\varrho$.
	In particular,
	\begin{equation}\label{eq:PiinBr}
		\Pi\mm_{1/2}^1(p_0,p_1)\in B_\varrho.
	\end{equation}

	Since $\eta(\dot\gamma)$ is constant, 
	\[
		z(\gamma(t))
		=
		z(p_0)+\int_0^t\la y,\dot x\ra_{\RR^m}\,ds+\mathrm{const}\cdot\,t,
		\qquad t\in[0,1].
	\]

	Evaluating at $t=\tfrac12$ and $t=1$, we get:
	\[
		z(\mm_{1/2}^1(p_0,p_1))-\frac{z(p_0)+z(p_1)}2
		=
		\frac12\int_0^{1/2}\la y,\dot x\ra_{\RR^m}\,ds
		-\frac12\int_{1/2}^{1}\la y,\dot x\ra_{\RR^m}\,ds.
	\]

	Since $\Pi\gamma\subseteq B_\varrho$ we have $|y|_{\RR^m}\le\varrho$ along $\gamma$,
	while $|\dot x|_{\RR^m}\le|\dot{(\Pi\gamma)}|_{\RR^{2m}}$; hence
	\[
		\left|z(\mm_{1/2}^1(p_0,p_1))-\frac{z(p_0)+z(p_1)}2\right|
		\le
		\frac12\int_0^1|y|_{\RR^m}\,|\dot x|_{\RR^m}\,ds
		\le
		\frac{\varrho}{2}\,\Len_{\RR^{2m}}(\Pi\gamma).
	\]

	Since $\Pi\gamma$ is a circular arc of central angle $\theta\in(-\pi,0)$ and radius $R=\tfrac{r}{2\sin(-\theta/2)}$,
	its Euclidean length is
	\[
		\Len_{\RR^{2m}}(\Pi\gamma)
		=
		R\,(-\theta)
		=
		\frac{-\theta/2}{\sin(-\theta/2)}\,r
		\le
		\frac\pi2\,r
		\le
		\pi\varrho,
	\]
	where we used $r\le2\varrho$ and $-\theta/2\in(0,\pi/2)$. It follows that
	\begin{equation}\label{eq:zininterval}
		\left|z(\mm_{1/2}^1(p_0,p_1))-\frac{z(p_0)+z(p_1)}2\right|
		\le
		\frac\pi2\varrho^2.
	\end{equation}

	Since $p_i \in A_i$ are arbitrary,~\eqref{eq:cylinder-z-defect} follows from~\eqref{eq:PiinBr} and~\eqref{eq:zininterval}.
	\end{proof}

\bibliographystyle{plain}
\bibliography{references}

\end{document}